\documentclass[a4paper,11pt]{amsart}
\usepackage{amssymb}
\usepackage{amscd}
\usepackage{comment}
\usepackage{amsmath,amsthm}
\usepackage[colorlinks=true]{hyperref}
\usepackage{enumerate}
\usepackage{booktabs,multirow}
\usepackage{tikz}
\usepackage{rotating}
\usetikzlibrary{patterns}
\usetikzlibrary{decorations.pathreplacing}
\usetikzlibrary{calc,through}

\allowdisplaybreaks[1]
\numberwithin{figure}{section}
\numberwithin{equation}{section}

\title{Rational Dyck tilings}
\author[K.~Shigechi]{Keiichi~Shigechi}
\email{k1.shigechi AT gmail.com}
\date{\today}

\newcommand\tikzpic[2]{
\raisebox{#1\totalheight}{
\begin{tikzpicture}
#2
\end{tikzpicture}
}}
\newtheorem{theorem}[figure]{Theorem}
\newtheorem{example}[figure]{Example}
\newtheorem{lemma}[figure]{Lemma}
\newtheorem{defn}[figure]{Definition}
\newtheorem{prop}[figure]{Proposition}

\newtheorem{remark}[figure]{Remark}
\begin{document}
\begin{abstract}
We introduce rational Dyck tilings, or $(a,b)$-Dyck tilings, 
and study them by the decomposition into $(1,1)$-Dyck tilings.
This decomposition allows us to make use of combinatorial 
models for $(1,1)$-Dyck tilings such as the Hermite history and 
the Dyck tiling strip bijection.
Together with $b$-Stirling permutations associated to 
the rational Dyck tilings, we obtain a correspondence 
between an $(a,b)$-Dyck tiling and a tuple of $ab$ $(1,1)$-Dyck tilings.
\end{abstract}

\maketitle

\section{Introduction}
A rational Dyck tiling, also called an $(a,b)$-Dyck tiling,  
is a tiling by rational Dyck tiles in the 
region above $\lambda$ and below $\mu$, where $\lambda$ and $\mu$ 
are rational Dyck paths satisfying $\lambda\le\mu$.
There are two types of rational Dyck tilings: one is a cover-inclusive
tiling, and the other is a cover-exclusive tiling.
In this paper, we study the rational Dyck tilings by decomposing them into 
a tuple of Dyck tilings.

Dyck tilings naturally appear in relation with the computation 
of the parabolic Kazhdan--Lusztig polynomials $P^{\pm}_{\lambda,\mu}$ for the maximal 
parabolic subgroups of type $A$ \cite{Bre02,SZJ12}.
The computation of $P^{-}_{\lambda,\mu}$ requires cover-exclusive tilings as shown in \cite{Bre02}.
Similarly, one make use of cover-inclusive tilings to compute $P^{+}_{\lambda,\mu}$ as in \cite{SZJ12}.
Cover-inclusive Dyck tiligns also appear in research areas in mathematical physics \cite{FN12,KKP17,KW11,KW15,K12,Pa19,PelWu19,Pon18}. 
There are several generalizations of Dyck tilings.
First generalization is to consider the Kazhdan--Lusztig polynomials of other types, and 
this leads to ballot tilings for type $B$ \cite{S17}, and Dyck tilings of type $D$ in \cite{S21a}.
Second is to impose a symmetry on Dyck tilings, and this leads to symmetric Dyck tilings 
studied in \cite{JVK16,S20}. Symmetric Dyck tilings have common properties of both type $A$ and type $B$.
Third is to change the structure of Dyck path. In other words, we consider $b$-Dyck 
paths or more generally $(a,b)$-Dyck paths. 
This generalization gives $b$-Dyck tilings studied in \cite{JVK16}, and $(a,b)$-Dyck 
tilings which are the main object in this paper.

Since we decompose $(a,b)$-Dyck tilings into $ab$ $(1,1)$-Dyck tilings, 
main tools to study $(a,b)$-Dyck tilings can be reduced to the tools 
used for $(1,1)$-Dyck tilings.
We mainly make use of two approaches for $(1,1)$-Dyck tilings: 
the Hermite history, and the DTS bijection \cite{KMPW12,S19}. 
Both approaches behave nicely with the inversion number of 
a Dyck tiling compared to other approaches.
One can use a rooted tree to describe a $(1,1)$-Dyck tiling (see for example \cite{KMPW12,S19,SZJ12}). 
Then, the labels on the edges of the tree are given by the DTS 
bijection for a $(1,1)$-Dyck tiling \cite{KMPW12,S19,S20}. 
The labels are strictly increasing from the root to leaves of 
the tree.
Then, the post-order word of the labels gives a permutation 
which characterizes the $(1,1)$-Dyck tiling together with the shape of the tree.
On the other hand, to characterize an $(a,b)$-Dyck tiling, we make use of 
a multi-permutation instead of a permutation. In fact, we consider a special class of 
multi-permutations, $b$-Stirling permutations \cite{GesSta78,Par94a,Par94b}.
This restriction of multi-permutations corresponds to capture 
the shape of the tree in case of $(1,1)$-Dyck tilings.
This is because that a $b$-Stirling permutation is one-to-one 
to the $(b+1)$-ary tree with labels (see for example \cite{CebGonDLe19}). 

To relate cover-inclusive $(a,b)$-Dyck tilings with cover-exclusive 
$(a,b)$-Dyck tilings, we introduce an incidence matrix which expresses 
the cover-exclusive tiligns. 
The parabolic Kazhdan--Lusztig polynomials $P^{\pm}_{\lambda,\mu}$ 
are dual to each other, {\it i.e.}, $P^{\pm}_{\lambda,\mu}$ 
is obtained from $P^{\pm}_{\lambda,\mu}$ by taking the transpose of the 
inverse of it.
Therefore, the cover-inclusive Dyck tilings are expressed in terms of 
the incidence matrix by taking the inverse.

We introduce two types of decompositions of $(a,b)$-Dyck paths: 
the horizontal and the vertical decompositions. 
The horizontal decomposition was defined for $(1,b)$-Dyck paths in \cite{KalMuh15}.
A decomposition of an $(a,b)$-Dyck tiling into $ab$ $(1,1)$-Dyck tilings
imposes some constraints on them.
These constraints give relations among $(1,1)$-Dyck tilings, which 
insure that we have a $b$-Stirling permutation for the $(a,b)$-Dyck 
tiling.
Besides those, we can distinguish a $(a,b)$-Dyck tiling with or without 
non-trivial $(a,b)$-Dyck tiles by looking at $(1,1)$-Dyck tilings.

By the horizontal decomposition of an $(1,b)$-Dyck tiling, 
we obtain $(1,1)$-Dyck tilings.
By combining the DTS bijection with the constraints on $(1,1)$-Dyck 
tilings, we can compute the weight of the $(1,b)$-Dyck tiling
through via $(1,1)$-Dyck tilings.
Further, one can show that the Hermite history of the $(1,b)$-Dyck tiling
is compatible with the Hermite histories of the $(1,1)$-Dyck tilings. 
We have similar results for the vertical decomposition.

The paper is organized as follows.
In Section \ref{sec:rDT}, we introduce the notions of $(a,b)$-Dyck paths 
and $(a,b)$-Dyck tilings. Then, we briefly review the Hermite history of a 
$(1,1)$-Dyck tiling and the Dyck tiling strip bijection.
We establish relations between cover-inclusive and cover-exclusive rational
Dyck tilings in Section \ref{sec:IMat}. 
In Section \ref{sec:decomp}, we decompose an $(a,b)$-Dyck tilings into 
$ab$ Dyck tilings by use of the horizontal and vertical decompositions and 
$b$-Stirling permutations. 
In Section \ref{sec:weight}, we give a description of the weight of 
an $(a,b)$-Dyck tiling in terms of words of $(1,1)$-Dyck tilings obtained by
the DTS bijections.

\section{Rational Dyck tilings}
\label{sec:rDT}
\subsection{Rational Dyck paths}
Let $(a,b)\in\mathbb{N}^{2}$ be relatively prime positive integers, and $n\in\mathbb{N}$.
A rational Dyck path of size $n$ is a lattice path from $(0,0)$ to $(bn,an)$, which 
does not go below the line $y=ax/b$.
We call a rational path an $(a,b)$-Dyck path when we emphasize $(a,b)$.
\begin{figure}[ht]
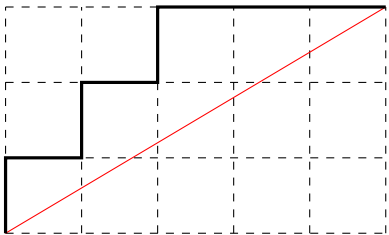

\tikzpic{-0.5}{
\draw[dashed](0,0)--(5,0)(0,1)--(5,1)(0,2)--(5,2)(0,3)--(5,3);
\draw[dashed](0,0)--(0,3)(1,0)--(1,3)(2,0)--(2,3)(3,0)--(3,3)(4,0)--(4,3)(5,0)--(5,3);
\draw[red](0,0)--(5,3);
\draw[very thick](0,0)--(0,1)--(1,1)--(1,2)--(2,2)--(2,3)--(5,3);
}
\caption{A rational Dyck path with $(a,b)=(3,5)$ and $n=1$.}
\label{fig:rDyck}
\end{figure}
Since an $(a,b)$-Dyck path $p$ is a lattice path, $p$ consists of north steps and east steps.
Here, north (resp. east) step means the vector $(0,1)$ (resp. $(1,0)$).
By assigning a north step $N$ and an east step $E$, we simply write the path $p$
as a word of $N$'s and $E$'s.

\begin{defn}
We denote by $\mathfrak{D}_{n}^{(a,b)}$ the set of $(a,b)$-Dyck paths 
of size $n$.
\end{defn}

Figure \ref{fig:rDyck} shows an example of a rational Dyck path with 
$(a,b)=(3,5)$ and $n=1$. 
The path is written as $NENENE^{3}$.

Since the region, which is surrounded by a rational Dyck path $p$, the line $x=0$, and 
the line $y=an$, can be regarded a Young diagram $\pi$, 
we identify $p$ with $\pi$ by abuse of notation.
In Figure \ref{fig:rDyck}, the path is identified with the Young diagram $(2,1)$.

Let $\lambda$ and $\mu$ be two rational Dyck paths.
Then, we say that $\mu$ is above $\lambda$ (or equivalently $\lambda$ 
is below $\mu$) if and only if the Young diagrams $\lambda$ and $\mu$ 
satisfy $\lambda\supseteq\mu$.
We denote by $\lambda\le\mu$ if $\mu$ is above $\lambda$.
Note that when $\mu$ is above $\lambda$, one can consider a skew shape 
$\lambda/\mu$.

We introduce a special class of $(a,b)$-Dyck paths which we call enlarged Dyck paths.
Let $p$ be a $(1,1)$-Dyck path expressed as a word of $N$'s and $E$'s.
We replace each $N$ and $E$ in $p$ by $N^{a}$ and $E^{b}$ and denote 
the new path by $p^{(a,b)}$.
Obviously, the path $p^{(a,b)}$ is a $(a,b)$-Dyck path.
We call the path $p^{(a,b)}$ an $(a,b)$-enlarged Dyck path.
If the size of a $(1,1)$-Dyck path $p$ is $n$, we define the size of $p^{(a,b)}$ is also $n$.

\begin{example}
we have two $(2,3)$-Dyck paths for $n=1$.
They are $N^{2}E^{3}$ and $NENE^{2}$.
The former is a $(2,3)$-enlarged Dyck path, but the latter 
is not. 
\end{example}

\subsection{Rational Dyck tilings}
Let $\lambda\le\mu$ be two rational Dyck paths.
A ribbon is a skew shape $\lambda/\mu$ which does not contain a two-by-tow box.
Then, an $(a,b)$-Dyck tile $d$ is defined as a ribbon such that 
the centers of boxes form an $(a,b)$-enlarged Dyck path $p^{(a,b)}$. 
We say that the tile $d$ is characterized by the path $p^{(a,b)}$ 
of size $n$.

Note that a single box is also an $(a,b)$-Dyck tile since a point 
is a rational Dyck path of size zero.

\begin{defn}
\label{defn:rDT}
Let $\lambda\le\mu$ be two rational $(a,b)$-Dyck paths.
A rational Dyck tiling (or a $(a,b)$-Dyck tiling) in the region $\lambda/\mu$ 
is a tiling by $(a,b)$-enlarged Dyck tiles.
\end{defn}

A box $(x,y)$ means a box whose center is $(x,y)$.
Let $b$ be a box $(x,y)$.
Then, a box $(x-1,y+1)$ is said to be NW (north-west) of $b$, a box $(x,y+1)$ is 
N of $b$, and a box $(x-1,y)$ is W of $b$.

In Definition \ref{defn:rDT}, we have no constraints on rational Dyck tiles.
Below, we consider the two special classes of rational Dyck tilings in the 
region $\lambda/\mu$.

We consider the following conditions on rational Dyck tiles.
Let $d_1$ and $d_2$ be rational Dyck tiles. 
\begin{enumerate}[(I)]
\item Then, if we move $d_1$ by 
$(1,-1)$, then it is contained by another Dyck tile $d_2$ or below 
the path $\lambda$.

\item If there exists a box of $d_1$ N, W, or NW of a box $d_2$,
then all boxes N, W, or NW of a box of $d_2$ belong to 
$d_1$ or $d_2$.
\end{enumerate}

Roughly speaking, the condition (I) means the sizes of rational Dyck tiles
are weakly decreasing from south-east to north-west direction.
On the other hand, the condition (II) means that the sizes of tiles 
are strictly increasing from south-east to north-west direction.

\begin{defn}
A rational Dyck tiling is said to be cover-inclusive (resp. cover-exclusive)
if and only if all rational Dyck tiles satisfy the condition (I) (resp. (II)).
\end{defn}

Figure \ref{fig:dt} shows two examples of cover-inclusive $(2,3)$-Dyck tilings. 
\begin{figure}[ht]
\tikzpic{-0.5}{[scale=0.5]
\draw[very thick](0,0)--(0,1)--(1,1)--(1,3)--(4,3)--(4,4)--(6,4);
\draw(0,1)--(0,3)--(1,3)(0,2)--(1,2);
\draw(2,3)--(2,4)--(4,4)(3,3)--(3,4);
\draw[red](0,0)--(6,4);
}
\qquad
\tikzpic{-0.5}{[scale=0.5]
\draw[very thick](0,0)--(0,1)--(1,1)--(1,3)--(4,3)--(4,4)--(6,4);
\draw(0,1)--(0,4)--(4,4);
\draw[red](0,0)--(6,4);
}
\caption{Two cover-inclusive $(2,3)$-Dyck tilings}
\label{fig:dt}
\end{figure}
The tiling in the left picture consists of four single boxes.
The tiling in the right picture contains a $(2,3)$-Dyck tile 
of size one.

We define three statistics $\mathrm{tiles}$, $\mathrm{area}$ and $\mathrm{art}$ 
on a rational Dyck tile $d$.

Suppose that a rational Dyck tile $d$ is characterized by an $(a,b)$-enlarged 
Dyck path of size $n$.
Then, we define
\begin{align*}
\mathrm{tiles}(d)&:=1, \\
\mathrm{area}(d)&:=an+bn+1,\\
\mathrm{art}(d)&:=(a\cdot\mathrm{tiles}(d)+b\cdot\mathrm{area}(d))/(a+b), \\
&=bn+1.
\end{align*}

\begin{remark}
The statistics $\mathrm{tiles}$ counts the number of tiles forming a rational Dyck tile $d$, which 
is one, $\mathrm{area}$ counts the number of boxes forming $d$.
One may define the statistics $\mathrm{art}$ by 
\begin{align*}
\mathrm{art}(d)&:=(b\cdot\mathrm{tiles}(d)+a\cdot\mathrm{area}(d))/(a+b), \\
&=an+1.
\end{align*}
Since we have a natural bijection between $(a,b)$-Dyck tilings and $(b,a)$-Dyck tilings
by reflecting the picture along the line $y=-x$, 
one can choose one of the definitions of $\mathrm{art}$
without loss of generality.
\end{remark}

Let $\mathcal{D}$ be a $(a,b)$-Dyck tiling. 
\begin{defn}
We define the weight of $\mathcal{D}$ as 
\begin{align*}
\mathrm{wt}(\mathcal{D})=\sum_{d\in\mathcal{D}}\mathrm{art}(d),
\end{align*}
where $d$ is a Dyck tile in $\mathcal{D}$.
\end{defn}

\begin{example}
The weights of two $(2,3)$-Dyck tilings in Figure \ref{fig:dt}
are both four.	
\end{example}

Let $\mathfrak{D}^{(a,b)}(\lambda)$ be the set of rational cover-inclusive Dyck 
tiling above $\lambda$.
\begin{defn}
The generating function $\mathfrak{Z}^{(a,b)}(\lambda)$  of rational Dyck 
tilings above $\lambda$ is
defined as 
\begin{align*}
\mathfrak{Z}^{(a,b)}(\lambda)
=\sum_{\mathcal{D}\in\mathfrak{D}(\lambda)}q^{\mathrm{wt}(\mathcal{D})},
\end{align*}
where $q$ is an indeterminate.
\end{defn}
\begin{remark}
The generating function $\mathfrak{Z}^{(a,b)}(\lambda)$ can not be 
expressed in terms of $q$-integers.
For example, we have
\begin{align*}
\mathfrak{Z}^{(2,3)}(NEN^2E^3NE^2)=1+2q+3q^2+3q^3+3q^4+q^5+q^6.
\end{align*}
However, $\mathfrak{Z}^{(1,1)}(\lambda)$ can be expressed in terms 
of $q$-integers in a simple form. See \cite{KW11,K12,KMPW12,S19} for details.
\end{remark}

\subsection{Hermite histories and Cover-inclusive Dyck tilings}
\label{sec:Hh}
In this subsection, we summarize some properties of cover-inclusive Dyck 
tiling above a Dyck path $\mathfrak{D}_{n}^{(1,1)}$ following \cite{KMPW12,S19}.

Let $p$ be a $(1,1)$-Dyck path and $\mathcal{D}(p)$ be a Dyck tiling
above $p$. 
We denote by $q$ the top path of $\mathcal{D}(p)$.
Let $d$ be a Dyck tile in $\mathcal{D}(p)$. 
We call the right-most edge of $d$ {\it entry} and 
the vertical edge at the bottom of $d$ {\it exit}.
We connect the entry and the exit of $d$ by a line.
We call this line a {\it trajectory}.
We concatenate trajectories of $\mathcal{D}(p)$ 
if and only if it the entry of a Dyck tile is attached 
to the exit of another Dyck tile.

Note that a concatenated trajectory may start from a $N$-step $s'_{N}$
is a Dyck path $p'$ which is above $p$.
In this case, we say this trajectory is attached to the 
$N$ step $s_{N}$ in $p$ such that $s_N$ is obtained from $s'_N$ 
by moving $s'_{N}$ in the $(1,-1)$-direction.
We say such a trajectory  is attached to the 
up step $N$ in $p$.
We have an obvious bijection between a trajectory and a $N$ step in $p$.
The set of trajectories is called an {\it Hermite history}.

Let $S_N$ be a $N$ step in $p$.
We denote by $l(S_N)$ the sum of the size of Dyck tiles and 
the number of Dyck tiles on the trajectory attached to the 
step $S_{N}$. 
Here, the size of a Dyck tile is the size of Dyck path 
characterizing this Dyck tile. 
Thus, a single box is a Dyck tile of size $0$.

We introduce a {\it chord} of a Dyck path $p$.
Since $p$ consists of $N$'s and $E$'s, and they are balanced, 
we make a pair of $N$ and $E$ next to each other in this order.
Then, by ignoring such pairs, we continue to make pairs.
A pair of $N$ and $E$ obtained in this way is called a chord 
of $p$.

We assign an integer in $[1,n]$ to a chord of $p$ as 
follows.
Let $S_{N}$ be a $N$ step $p$.
We assign $l(S_{N})$ to the chord containing $S_{N}$.

We will define a permutation $\omega'(p)$ from the labeled chords in $p$.
We read the labels on chords of $p$ in the pre-order.
Here, pre-order means that we read the labels from 
the left-most and bottom-most chord, then left chords, and 
right chords. We continue this process until we read all 
the labels on the chords. 
We denote by $\omega(p)$ the inverse of 
the permutation obtained as above.
We say that $\omega(p)$ is obtained by an Hermite history.		

\begin{remark}
The labels on a chord are increasing from upper-left to 
bottom-right for a Dyck tiling obtained by an Hermite history.	
\end{remark}

\begin{example}
\label{ex:Hh}
We consider the two Dyck tilings associated 
to Dyck paths in $\mathfrak{D}_{4}^{(1,1)}$ as below.
\begin{align*}
\mathcal{D}_1=\tikzpic{-0.5}{[scale=0.5]
\draw[very thick](0,0)--(0,1)--(1,1)--(1,3)--(3,3)--(3,4)--(4,4);
\draw(0,1)--(0,4)--(3,4);
\draw[red](3,3.5)--(0.5,3.5)--(0.5,1.5)--(0,1.5);
\draw[gray](0,0)--(1,1)(1,1)--(4,4)(1,2)--(2,3);
\draw(0.5,0.5)node{$3$}(1.5,2.5)node{$1$}
(2,2)node{$2$}(3.5,3.5)node{$4$};
},\quad
\mathcal{D}_2=\tikzpic{-0.5}{[scale=0.5]
\draw[very thick](0,0)--(0,3)--(2,3)--(2,4)--(4,4);
\draw(0,3)--(0,4)--(2,4)(1,4)--(1,3);
\draw[red](2,3.5)--(0,3.5);
\draw[gray](0,2.5)--(0.5,3)(0,1.5)--(1.5,3)
(0,0.5)--(3.5,4)(2,3.5)--(2.5,4);
\draw(1.75,2.25)node{$4$}(0.25,2.75)node{$1$}
(0.75,2.25)node{$2$}(2.25,3.75)node{$3$};
}
\end{align*}
The pre-order words for these two Dyck tilings are 
$3214$ and $4213$. 
Thus, $\omega(\mathcal{D}_1)=3214$ and $\omega(\mathcal{D}_2)=3241$.
\end{example}

Let $w$ be a permutation on the alphabets $[1,n]$.
We construct a non-negative integer sequence 
$\mathfrak{h}(w):=(h_1,\ldots,h_n)$ as follows.
Let $w_i$ be a permutation consisting of integers $[1,i]$ in $w$.
We define $h_{j}$ by the position of $j$ in $w_j$ from left 
minus one.
For example, we have $\mathfrak{h}(w)=(0,1,0,2)$ if $w=3142$.
\begin{defn}
\label{defn:insh}
We call $\mathfrak{h}(w)$ the insertion history of a permutation $w$.
\end{defn}

Let $\mathfrak{u}(p)$ be a step sequence of the Dyck path $p$, 
and $\mathcal{D}$ be a Dyck tiling above $p$ without non-trivial
Dyck tiles.
We denote by $q$ the top Dyck path in $\mathcal{D}$.

\begin{prop}
\label{prop:ssforq}
The integer sequence $\mathfrak{u}(p)-\mathfrak{h}(\omega(\mathcal{D}))$ 
is the step sequence of $q$.
\end{prop}
\begin{proof}
We prove the statement by induction. 
We first consider a unique Dyck tiling $\mathcal{D}_{0}$ such that the top path 
and the bottom path are both $p$.
By construction of an Hermite history, $\omega_0:=\omega(\mathcal{D}_{0})$ is 
the permutation with maximal inversions, {\it i.e.}, 
$\omega_{0}=(n,n-1,\ldots,1)$.
The insertion history for $\omega_0$ is $(0,\ldots,0)$, which implies
$\mathfrak{u}(p)-\mathfrak{h}(\omega_0)=\mathfrak{u}(p)$. 
Since $q=p$, $\mathfrak{u}(p)$ gives the step sequence of $q$.

We assume that the statement is true up to all path $q'$ below $q$.
The Dyck tiling $\mathcal{D}$ above $p$ and below $q$ is obtained from some Dyck 
tiling $\mathcal{D}'$ above $p$ and below $q'$ by adding a single box.
This addition comes from the fact that $\mathcal{D}$ and $\mathcal{D}'$ have 
no non-trivial Dyck tiles.
The addition of boxes to $\mathcal{D}'$ implies the following operation 
on the labels on chords of $p$.
The addition of a single box results in an extension of a trajectory 
by a single box. Without loss of generality, we assume that the trajectory
associated to a chord $c_1$ is extended.
Let $l$ be the label on $c_1$ for $\mathcal{D}'$. 
Then, addition of a single box is equivalent to changing the label $l$ on $c_1$ 
and a label $l'$ on some $c_2$ such that the chord $c_2$ is left to $c_1$ and 
$l'$ is minimal label larger than $l$.
Since we have no non-trivial Dyck tiles, the chord $c_2$ should be left next to $c_{1}$.
Thus, the permutation $\omega(\mathcal{D})$ is obtained from $\omega(\mathcal{D}')$ 
by exchanging some integer $m$ and $m+1$. 
The condition that labels on chords are increasing from a left-top chord to 
a right-bottom chord insures that there is no integer $m'<m$ between 
$m$ and $m+1$ in $\omega(\mathcal{D})$. 
Then, the exchange of $m$ and $m+1$ in $\omega(\mathcal{D}')$ means that 
the $m+1$-th entry of $\mathfrak{h}(\omega(\mathcal{D}'))$ is increased 
by one.
From these observations, the step sequence of $q$ is also expressed 
as $\mathfrak{u}(p)-\mathfrak{h}(\omega(\mathcal{D}))$.
This completes the proof.
\end{proof}

\subsection{Dyck tiling strip and cover-inclusive Dyck tilings}
\label{sec:DTS}
In Section \ref{sec:Hh}, we introduce an Hermite history, and 
construct a permutation $\omega(p)$ from a Dyck tiling $p$.
In this subsection, we introduce another construction of 
a permutation, called {\it Dyck tiling strip} (DTS for short) 
following \cite{KMPW12,S19}. 

Let $p$ be a Dyck path of size $n-1$ and $\mathcal{D}$ be 
a Dyck tiling above $p$.
First, we define an operation called {\it spread} of $\mathcal{D}$ 
at $y=-x+m$ for $0\le m\le 2(n-1)$.
We divide $\mathcal{D}$ into two pieces by the line $y=-x+m$.
Then, we move the up-right piece by $(1,1)$-direction.
By reconnecting two pieces by $NE$ steps, we obtain 
a new Dyck tiling of size $n$.

In a spread, if the line $y=-x+m$ passes through a Dyck tile,
the size of the Dyck tile is increased by one.
Figure \ref{fig:sp} is an example of a spread of a Dyck tiling
at $y=-x+3$. 
\begin{figure}[ht]
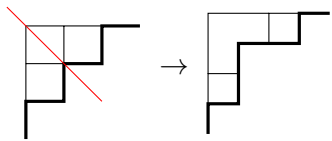

\begin{align*}
\tikzpic{-0.5}{[scale=0.5]
\draw[very thick](0,0)--(0,1)--(1,1)--(1,2)--(2,2)--(2,3)--(3,3);
\draw(0,1)--(0,3)--(2,3)(0,2)--(1,2)--(1,3);
\draw[red](-0.5,3.5)--(2,1);
}\rightarrow
\tikzpic{-0.5}{[scale=0.4]
\draw[very thick](0,0)--(0,1)--(1,1)--(1,3)--(3,3)--(3,4)--(4,4);
\draw(0,1)--(0,2)--(1,2)(2,3)--(2,4)--(3,4)(0,2)--(0,4)--(2,4);
}
\end{align*}
\caption{A spread at $y=-x+3$ of a Dyck tiling}
\label{fig:sp}
\end{figure}

Let $\mathcal{D}'$ be a new Dyck tiling of size $n$ after 
a spread of $\mathcal{D}$.
We perform an addition of boxes as follows.
We attach a single box to each $N$ step in $\mathcal{D}$ such 
that the $N$ step is right to the line $y=-x+m+1$.
We denote by $\mathcal{D}_{new}$ the new Dyck tiling 
obtained from $\mathcal{D}'$.
This process is called {\it right strip-growth}.
Similarly, if we attach boxes to each $E$ step in $\mathcal{D}$
such that the $E$ step is left to the line $y=-x+m+1$.
This process is called {\it left strip-growth}.

Since a spread increases the number of chords for the path $p$
by one, we have a new added chord. 
Then, we assign a label $n$ to the newly added chord.
In this way, we obtain labels on chords for a Dyck tiling $\mathcal{D}_{new}$.
We call these processes to obtain $\mathcal{D}_{new}$ 
from $\mathcal{D}$ {\it Dyck tiling strip}. 

\begin{defn}
We say Dyck tiling strip is right (resp. left) Dyck tiling strip if we perform 
right (resp. left) strip-growths.
If we do not specify right or left, a Dyck tiling strip means right Dyck tiling 
strip.
\end{defn}

Given a Dyck tiling $\mathcal{D}$, we have a label on chords 
of $p$.
We denote by $\nu(\mathcal{D})$ the word read by the post-order.
Here, post-order means that we read the labels of chords by the 
following order: 1) the left chords, 2) the right chords, then 
3) the bottom chords.
\begin{example}
\label{ex:DTS}
We consider the same Dyck tilings as Example \ref{ex:Hh}.
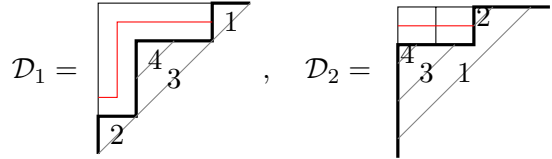
\begin{align*}
\mathcal{D}_1=\tikzpic{-0.5}{[scale=0.5]
\draw[very thick](0,0)--(0,1)--(1,1)--(1,3)--(3,3)--(3,4)--(4,4);
\draw(0,1)--(0,4)--(3,4);
\draw[red](3,3.5)--(0.5,3.5)--(0.5,1.5)--(0,1.5);
\draw[gray](0,0)--(1,1)(1,1)--(4,4)(1,2)--(2,3);
\draw(0.5,0.5)node{$2$}(1.5,2.5)node{$4$}
(2,2)node{$3$}(3.5,3.5)node{$1$};
},\quad
\mathcal{D}_2=\tikzpic{-0.5}{[scale=0.5]
\draw[very thick](0,0)--(0,3)--(2,3)--(2,4)--(4,4);
\draw(0,3)--(0,4)--(2,4)(1,4)--(1,3);
\draw[red](2,3.5)--(0,3.5);
\draw[gray](0,2.5)--(0.5,3)(0,1.5)--(1.5,3)
(0,0.5)--(3.5,4)(2,3.5)--(2.5,4);
\draw(1.75,2.25)node{$1$}(0.25,2.75)node{$4$}
(0.75,2.25)node{$3$}(2.25,3.75)node{$2$};
}
\end{align*}
The post-order words for $\mathcal{D}_1$ and $\mathcal{D}_2$
are $\nu(\mathcal{D}_1)=2431$ and $\nu(\mathcal{D}_2)=4321$ respectively.
\end{example}

\begin{remark}
The labels on chords are decreasing from left-top to bottom-right 
for a Dyck tiling obtained by the DTS bijection.
\end{remark}

Labels on chords in an Hermite history and those in the DTS bijection
are related as below.
\begin{prop}
\label{prop:dualHhDTS}
Let $\mathcal{D}$ be a Dyck tiling above $p$, and 
$h_1(c)$ (resp. $h_2(c)$) are a label of a chord $c$ in $p$ 
obtained from $\mathcal{D}$ for an Hermite history 
(resp. DTS bijection).
Then, we have 
\begin{align}
\label{eqn:dualHhDTS}
h_1(c)+h_{2}(c)=n+1.
\end{align}
\end{prop}
\begin{proof}
We prove the proposition by induction.
For a Dyck tiling $\mathcal{D}$, we denote by $q$ the top 
path.
The statement is obvious when $p=q$.
We assume that Eqn. (\ref{eqn:dualHhDTS}) holds for all $q'$ 
below $q$.

As in the proof of Proposition \ref{prop:ssforq}, we extend a 
trajectory by length one in the Hermite history.
On the other hand, the strip-growth in the DTS bijection corresponds to 
extend a trajectory by one.
Further, the labels on chords in the Hermite history are increasing 
from left-top chords to right-bottom chords, while the labels in the 
DTS bijection are increasing from left-top to right-bottom.
Thus, the extension of a trajectory increases a label in the Hermite 
history and decrease the corresponding label in the DTS bijection. 
From these observations, we have Eqn. (\ref{eqn:dualHhDTS}).
\end{proof}

\section{Incidence matrix}
\label{sec:IMat}
Let $\pi:=\pi_1\ldots\pi_{r}$ be a path of length $r$ consisting of $N$'s and $E$'s.
Suppose that $\pi_{i}=N$ and $\pi_{j}=E$ for $1\le i<j\le r$. 
We say that $(i,j)$ is $(a,b)$-admissible if and only if
the partial path $\pi':=\pi_{i+1}\ldots\pi_{j-1}$ 
is an $(a,b)$-enlarged Dyck path, and 
$\pi'$ cannot be written as a concatenation of two $(a,b)$-enlarged 
Dyck paths.

Suppose that $\pi_{i}=N$ in $\pi$. Then, by definition of 
$(a,b)$-admissibility, we have at most one $j$ such that 
the pair $(i,j)$ is $(a,b)$-admissible.
Note that if $\pi_{i}=N$ and $\pi_{i+1}=E$, the pair $(i,i+1)$ is 
$(a,b)$-admissible for any pair of $a$ and $b$.

We define the operation, which we call $NE$-flipping, on a path $\pi$.
Suppose that the pair $(i,j)$ is $(a,b)$-admissible in $\pi$.
The $NE$-flip is defined to exchange $\pi_i$ and $\pi_j$, namely, 
we have a new path $\pi'$ such that 
$\pi'_i=E$, $\pi'_j=N$ and $\pi'_{k}=\pi_{k}$ for $k\neq i,j$.
We denote this relation by $\pi'\leftarrow\pi$, or equivalently
$\pi\rightarrow\pi'$.

Let $\mathcal{F}(\pi)$ be the set of $(a,b)$-admissible pairs in $\pi$, and 
$f_{i}\in\mathcal{F}(\pi)$ for $1\le i\le |\mathcal{F}(\pi)|$.
We consider the sequence 
\begin{align}
\label{eqn:chainflip}
\pi\stackrel{i_1}{\rightarrow}\pi_{1}\stackrel{i_2}{\rightarrow}\pi_{2}\ldots
\stackrel{i_p}{\rightarrow}\pi_{p},
\end{align} 
with $1\le i_1<i_2<\ldots<i_p\le|\mathcal{F}(\pi)|$.
We denote by $\pi_{j}\stackrel{i_{j+1}}{\rightarrow}\pi_{j+1}$ the 
$NE$-flip by the pair$f_{i_{j+1}}$.

As already mentioned above, we have no chance to have two $(a,b)$-admissible 
pairs $(i,j)$ and $(i,j')$ with $j\neq j'$ or $(i,j)$ and $(i',j)$ with $i\neq i'$.
Thus, we can choose the order of $i_j$ with $1\le j\le p$ as above. 

Given two paths $\pi$ and $\pi'$, we denote by $\pi\rightarrow \pi'$ 
if there exists a sequence (\ref{eqn:chainflip}) with $\pi_{p}=\pi'$.

We introduce two types of weight which is given to a $NE$-flip.
We define the weight of type $I$ of a $NE$-flip as $\mathrm{wt}^{I}(\pi'\leftarrow\pi)=-q$ where 
$\pi'$ is obtained by a single $NE$-flip from $\pi$.
Then, in general, we define $\mathrm{wt}^{I}(\pi'\leftarrow\pi)=(-q)^{p}$ when 
$\pi'\leftarrow\pi$ as a sequence (\ref{eqn:chainflip}).

Let $(i,j)\in\mathcal{F}(\pi)$. 
Then, we denote by $n_{E}$ the number of $E$'s  in the partial 
path $\pi_{i+1}\ldots\pi_{j-1}$.
The weight of type $II$ of a $NE$-flip is given by 
$\mathrm{wt}^{II}(\pi'\leftarrow\pi)=-q^{n_{E}}$.
For general two paths $\pi$ and $\pi'$, the weight 
$\mathrm{wt}^{II}(\pi'\leftarrow\pi)$ is given by the 
product of the weights of $NE$-flips.

For both types $I$ and $II$, the weight $\mathrm{wt}^{X}(\pi'\leftarrow\pi)$ with $X=I$ or $II$ 
is given by the monomial of $q$.
When there exists no sequence (\ref{eqn:chainflip}) starting from $\pi$ and ending with $\pi'$,
we define $\mathrm{wt}^{X}(\pi'\leftarrow\pi)=0$. 

We define two types of incidence matrices as follows.

\begin{defn}
The matrices $M$ and $N$ are defined by 
\begin{align*}
M_{\pi',\pi}&=\mathrm{wt}^{I}(\pi'\leftarrow\pi),\\
N_{\pi',\pi}&=\mathrm{wt}^{II}(\pi'\leftarrow\pi).
\end{align*}
\end{defn}

\begin{prop}
The entry $M_{\pi',\pi}$ (resp. $N_{\pi',\pi}$)
gives the cover-exclusive $(a,b)$-Dyck tiling in the region 
above $\pi'$ and below $\pi$.
The weight of a tile is given by the statistics $\mathrm{tiles}$ 
(resp. $\mathrm{art}$).
\end{prop}
\begin{proof}
By definition, the entry $\pi'$ is obtained from $\pi$ by exchanging 
an $N$ and an $E$ in $\pi$. If we connect a pair of these $N$ and $E$ by 
a arc in $\pi$, arcs never intersect. In terms of a Dyck tiling,  
it is clear that $\pi'$ is obtained from $\pi$ by a cover-exclusive tiling.
Since the weight of $NE$-flip is a monomial of $q$, 
the weight of a tiling above $\pi'$ and below $\pi$ is given by the statistic 
$\mathrm{tile}$.

We have a similar proof for $N_{\pi',\pi}$. 
Since the weight of $NE$-flip is equal to the statistic $\mathrm{art}$, 
the weight of a tiling above $\pi'$
and below $\pi$ is given by $\mathrm{art}$.
This completes the proof.
\end{proof}

The inverse matrices of $M$ and $N$ are characterized by $(a,b)$-Dyck tilings
as follows.
\begin{theorem}
The entry $M^{-1}_{\pi',\pi}$ (resp. $N^{-1}_{\pi',\pi}$)
gives the cover-inclusive $(a,b)$-Dyck tiling in the region 
above $\pi'$ and below $\pi$.
The weight of a tile is given by the statistics $\mathrm{tiles}$ 
(resp. $\mathrm{art}$).
\end{theorem}
\begin{proof}
The theorem follows from the Principle of Inclusion-Exclusion in \cite{Sta97}.
See also the proof of Theorem  in \cite{SZJ12}.
\end{proof}

\begin{example}
We consider rational $(1,2)$-Dyck tilings. 
The order of bases is 
$NNEEE$, $NENEE$, $NEENE$, $NEEEN$, $ENNEE$, $ENENE$, $ENEEN$, $EENNE$, $EENEN$, and $EEENN$.
\begin{align*}
M&=
\begin{pmatrix}
1 & 0 & 0 & 0 & 0 & 0 & 0 & 0 & 0 & 0 \\
-q & 1 & 0 & 0 & 0 & 0 & 0 & 0 & 0 & 0 \\
0 & -q & 1 & 0 & 0 & 0 & 0 & 0 & 0 & 0 \\
0 & 0 & -q & 1 & 0 & 0 & 0 & 0 & 0 & 0 \\
0 & -q & 0 & 0 & 1 & 0 & 0 & 0 & 0 & 0 \\
0 & q^2 & -q & 0 & -q & 1 & 0 & 0 & 0 & 0 \\
-q & 0 & q^2 & -q & 0 & -q & 1 & 0 & 0 & 0 \\
0 & 0 & 0 & 0 & 0 & -q & 0 & 1 & 0 & 0 \\
q^2 & 0 & 0 & 0 & 0 & q^2 & -q & -q & 1 & 0 \\
0 & 0 & 0 & 0 & 0 & 0 & 0 & 0 & -q & 1 \\
\end{pmatrix},\\
M^{-1}&=
\begin{pmatrix}
1 & 0 & 0 & 0 & 0 & 0 & 0 & 0 & 0 & 0 \\
q & 1 & 0 & 0 & 0 & 0 & 0 & 0 & 0 & 0 \\
q^2 & q & 1 & 0 & 0 & 0 & 0 & 0 & 0 & 0 \\
q^3 & q^2 & q & 1 & 0 & 0 & 0 & 0 & 0 & 0 \\
q^2 & q & 0 & 0 & 1 & 0 & 0 & 0 & 0 & 0 \\
q^3 & q^2 & q & 0 & q & 1 & 0 & 0 & 0 & 0 \\
q+q^4 & q^3 & q^2 & q & q^2 & q & 1 & 0 & 0 & 0 \\
q^4 & q^3 & q^2 & 0 & q^2 & q & 0 & 1 & 0 & 0 \\
q^5 & q^4 & q^3 & q^2 & q^3 & q^2 & q & q & 1 & 0 \\
q^6 & q^5 & q^4 & q^3 & q^4 & q^3 & q^2 & q^2 & q & 1 \\
\end{pmatrix}
\end{align*}
Note that when $\pi'=ENEEN$ and $\pi=NNEEE$, we have a non-trivial 
Dyck tiling and this tiling contributes as $q$ in $M^{-1}$.
This tiling also contributes as $q^{3}$ in $N^{-1}$.
\end{example}

\section{Decomposition}
\label{sec:decomp}
Let $\mathcal{D}$ be a $(a,b)$-Dyck tiling.
The main purpose of this section is to introduce a decomposition 
of $\mathcal{D}$ into $ab$ $(1,1)$-Dyck tilings.
In \cite{KalMuh15}, a strip decomposition of $(1,b)$-Dyck paths is introduced
and studied.
The strip decomposition assigns a $(1,b)$-Dyck path to $b$ $(1,1)$-Dyck 
paths.
We call this strip decomposition the {\it horizontal} decomposition.
When $a\neq1$, one can also consider a strip decomposition which 
is in the {\it vertical} direction.
Similarly, a vertical decomposition gives a $(a,1)$-Dyck path 
to $a$ $(1,1)$-Dyck paths.
Thus, for $(a,b)$-Dyck paths, we can introduce the decomposition which 
is both horizontal and vertical.
This decomposition assigns a $(a,b)$-Dyck path to $ab$ $(1,1)$-Dyck paths.
In the language of Dyck tilings, $(a,b)$-Dyck paths correspond to 
a $(a,b)$-Dyck tiling consisting of single boxes, which means that 
the tiling do not contain non-trivial $(a,b)$-Dyck tiles.
In this section, we generalize this decomposition from $(a,b)$-Dyck 
paths to $(a,b)$-Dyck tilings.

\subsection{Step and height sequences}
Let $\pi$ be an $(a,b)$-Dyck path of size $n$.
We introduce the step sequence and the height sequence to the path $\pi$
following \cite{KalMuh15}.

The step sequence $\mathbf{u}(\pi)=(u_1,\ldots,u_{an})$ is a sequence 
of non-negative integers defined for $\pi$ as 
\begin{align*}
&u_1\le u_2\le \ldots \le u_{an}, \\
&u_{k}\le b(k-1)/a, \quad \forall k\in[1,an].
\end{align*}
The entry $u_{k}$ in the step sequence indicates that the path $\pi$
passes through the edge connecting $(u_{k},k-1)$ and $(u_k,k)$.

Similarly, the height sequence $\mathbf{h}(\pi)$ is a sequence of 
positive integers satisfying 
\begin{align*}
&h_1\le h_2\le \ldots\le h_{bn}, \\
&h_{k}\ge\lceil ka/b\rceil, \quad \forall k\in[1,bn].
\end{align*}
The entry $h_{k}$ in the height sequence indicates that the path $P$ 
passes through the edge connecting $(k-1,h_{k})$ and $(k,h_{k})$.

\begin{example}
Let $(a,b)=(2,3)$, $n=2$ and $\pi=NENENEENEE$.
The step sequence is $\mathbf{u}(\pi)=(0,1,2,4)$ and 
the height sequence is $\mathbf{h}(\pi)=(1,2,3,3,4,4)$.
\end{example}

\subsection{Horizontal and vertical strip decomposition}
Let $\pi$ be an $(a,b)$-Dyck path of size $n$ and 
$\mathbf{h}(\pi):=(h_1,h_2,\ldots,h_{bn})$ be its 
height sequence.
We construct $b$ integer sequences $\mathbf{H}_{i}:=(H^{i}_{1},\ldots,H^{i}_{n})$, $1\le i\le b$, 
from $\mathbf{h}(\pi)$ by
\begin{align*}
H^{i}_{j}:=h_{(j-1)b+i}, \quad j\in[1,n].
\end{align*}
The integer sequence $\mathbf{H}_{i}$ defines a lattice path $p_{i}$ from 
$(0,0)$ to $(n,an)$.
Note that a path $p_{i}$ may not be an $(a,1)$-Dyck path of size $n$.

\begin{defn}
\label{defn:hsdecomp}
We call the map from $\pi$ to $b$ paths $p_{i}$, $1\le i\le b$ 
the horizontal strip decomposition of $\pi$.
We denote this map by $\theta_{h}:\pi\mapsto (p_1,\ldots,p_{b})$.
\end{defn} 

Let $\mathbf{u}(\pi):=(u_1,\ldots,u_{an})$ be the step sequence 
of $\pi$.
We construct $a$ integer sequences $\mathbf{U}_{i}:=(U^{i}_{1},\ldots,U^{i}_{n})$, 
$1\le i\le a$, from $\mathbf{u}(\pi)$ by
\begin{align*}
U^{i}_{j}:=u_{(j-1)a+i}, \quad j\in[1,n]
\end{align*}
The integer sequence $\mathbf{U}_{i}$ defines a lattice path $q_{i}$ from 
$(0,0)$ $(bn,n)$.

\begin{defn}
\label{defn:vsdecomp}
We call the map from $\pi$ to $a$ paths $q_{i}$, $1\le i\le a$ 
the vertical strip decomposition of $\pi$.
We denote this map by $\theta_{v}:\pi\mapsto (q_1,\ldots,q_{a})$.
\end{defn} 

We give an another description of the horizontal and vertical 
decompositions of $\pi$.
We first consider the horizontal decomposition.
We define $b$ non-negative integer sequences $\mathbf{v}_{i}$, $1\le i\le b$, recursively 
as follows.

\noindent {\bf Algorithm A:}
\begin{enumerate}
\item Set $i:=b$, $\mathbf{u}':=\mathbf{u}(\pi)$ and $\mathbf{v}_{0}=(0,\ldots,0)$. 
\item Define
\begin{align}
\label{eqn:hsdandv}
\mathbf{v}_{b+1-i}:=\left\lceil \mathbf{u}'/i \right\rceil.
\end{align}
\item Decrease $i$ by one. Replace $\mathbf{u}'$ by $\mathbf{u}'-\mathbf{v}_{b+1-i}$.
Then, go to (2). The algorithm stops when $i=1$.
\end{enumerate}

\begin{prop}
\label{prop:vssp}
Let $\mathbf{v}_{i}$, $1\le i\le b$, be integer sequences defined as above.
Then, $\mathbf{v}_{i}$ is the step sequence of a Dyck path $p_{i}$.
\end{prop}
\begin{proof}
By a horizontal decomposition of $\pi$, we have $b$ lattice paths $p_i$, 
$1\le i\le b$.
The $i$-th Dyck path $p_i$ consists of $(j-1)b+i$-th columns for $1\le j\le n$.
Thus, it is straight forward that $\mathbf{v}_{i}$ in Eqn. (\ref{eqn:hsdandv}) 
is the step sequence of $p_{i}$.
\end{proof}

Given a partition $\lambda:=(\lambda_1,\ldots,\lambda_{l})$ where $l$ is 
the length of $\lambda$, we define 
the transposition of $\lambda$, denoted 
by $\lambda^{t}:=(\lambda^{t}_{1},\ldots,\lambda_{\lambda_1}^{t})$, as 
\begin{align*}
\lambda^{t}_{i}:=\#\{j | \lambda_{j}\ge i\}.
\end{align*}
We append several $0$'s to $\lambda^{t}_{i}$ if necessary.
For example, we consider a rational Dyck path in $\mathfrak{D}_{2}^{(2,3)}$.
If we have a step sequence $(0,1,2,4)$, then its transposition is 
given by $(3,2,1,1,0,0)$.

We consider the vertical decomposition of $P$.
We replace $\mathbf{u}(\pi)$ by its transposition $\mathbf{u}(\pi)^{t}$, 
$\mathbf{v}_{i}$ by $\mathbf{w}_{i}$ in Algorithm A.

\begin{prop}
\label{prop:transss}
Let $\mathbf{w}_{i}$, $1\le i\le b$, be integer sequences defined as above.
Then, $\mathbf{w}_{i}^{t}$ is the step sequence of a Dyck path $q_{i}$.
\end{prop}
\begin{proof}
Let $P=P_1P_2\ldots P_{(a+b)n}\in\mathfrak{D}_{n}^{(a,b)}$ be 
a rational Dyck path. 
We define $N^{\sharp}=E$ and $E^{\sharp}=N$.
The vertical decomposition of $P$ is equivalent to the horizontal 
decomposition of the $P^{\sharp}\in\mathfrak{D}_{n}^{(b,a)}$ 
where $P^{\sharp}$ is the Dyck path 
$P^{\sharp}=p_{(a+b)n}^{\sharp}\ldots p_1^{\sharp}$.

Note that taking the transposition means we apply the operation $\sharp$.
Then, the statement in the Proposition is a direct consequence of Proposition \ref{prop:vssp}
by the transposition.
\end{proof}

The following proposition is clear from the definitions of 
$\theta_{h}$ and $\theta_{v}$.
\begin{prop}
\label{prop:hvorder}
Let $\pi$ be an $(a,b)$-Dyck path.
Then, we have 
\begin{align*}
\theta_{h}\circ\theta_{v}=\theta_{v}\circ\theta_{h}.
\end{align*}
\end{prop}

\begin{defn}
We define $\vartheta:=\theta_{v}\circ\theta_{h}$.
We call $\vartheta$ Dyck path decomposition.
\end{defn}
The Dyck path decomposition $\vartheta$ sends an $(a,b)$-Dyck path 
of size $n$
to a set of $ab$ paths of size $n$.

\begin{example}
\label{ex:hvdecomp}
We consider a Dyck path $P=NENENE^2NE^2\in\mathfrak{D}_{2}^{(2,3)}$.
The actions of $\theta_h$ and $\theta_v$ on $P$ are given by
\begin{align*}
&\theta_{h}(P)=(NEN^2EN,N^2EN^2E,N^3ENE), \\
&\theta_{v}(P)=
\left(
\begin{matrix}
NE^2NE^4\\
ENE^3NE^2
\end{matrix}
\right)
\end{align*}
Then, the map $\vartheta$ gives $6$ lattice paths
\begin{align*}
\vartheta(P)=
\left(
\begin{matrix}
NENE & NENE & NNEE \\
ENEN & NENE & NENE
\end{matrix}
\right)
\end{align*}
Note that paths in each column (resp. row) of $\vartheta(P)$ are obtained 
by the vertical (resp. horizontal) decomposition of $\theta_h(P)$ 
(resp. $\theta_v(P)$).
\end{example}

\subsection{\texorpdfstring{$b$}{b}-Stirling permutations}
A $b$-Stirling permutation of size $n$ is a permutation of 
the multiset $\{1^b,2^b,\ldots,n^b\}$ such that 
if an integer $j$ appears between two $i$'s, we have $j>i$.

\begin{defn}
We denote by $\mathfrak{S}_{n}^{(b)}$ the set of $b$-Stirling 
permutations of size $n$.
\end{defn}

We construct a $b$-Stirling permutation from 
a non-negative integer sequence $\mathbf{u}:=(u_1,\ldots,u_{n})$ 
such that $u_i\le b(i-1)$ for all $1\le i\le n$.
Since $u_1=0$ for any integer sequence $\mathbf{u}$, we put $b$ $1$'s in line.
By definition, $\mathfrak{p}_{1}:=1^b$ is a $b$-Stirling permutation.
Let $\mathfrak{p}_{i}$ be a $b$-Stirling permutation consisting of integers 
in $[1,i]$.
The multi-permutation $\mathfrak{p}_{i}$ can be recursively obtained from 
$\mathfrak{p}_{i-1}$ by inserting $i^{b}$ into the $u_{i}$-th 
position in $\mathfrak{p}_{i-1}$.

\begin{defn}
\label{defn:mu}
Let $\mathbf{u}$ and $\mathfrak{p}_{n}$ be sequences as above.
We denote by $\mu: \mathbf{u}\mapsto \mathfrak{p}_{n}$ the 
map from a non-negative integer sequence to a $b$-Stirling 
permutation.
We call $\mathfrak{p}_n$ the insertion history of $\mathbf{u}$.
\end{defn}

\begin{remark}
In Definition \ref{defn:insh}, we introduce an insertion history
for a permutation. 
The map $\mu$ is a generalization of an insertion history for 
a $b$-Stirling permutation.	
\end{remark}

It is obvious that $\mu$ has an inverse from $\mathfrak{p}_{n}$ 
from $\mathbf{u}$. 
We denote by $\mu^{-1}$ the inverse of $\mu$.

\begin{example}
Let $\mathbf{u}=(0,1,2,4)$ and $b=3$. 
We have a sequence of $b$-Stirling permutations:
\begin{align*}
111\rightarrow 1222111 \rightarrow 1233322111 \rightarrow 1233444322111.
\end{align*}
\end{example}

We give a description of the Dyck path decomposition in terms of 
$b$-Stirling permutations.
We first consider the horizontal decomposition of a rational 
Dyck path, then consider the vertical decomposition. 
Note that the order of horizontal and vertical decompositions
is irrelevant to the result by Proposition \ref{prop:hvorder} 

Let $P$ be a rational Dyck path in $\mathfrak{S}_{n}^{(a,b)}$ and 
$\mathbf{u}(P)$ be its step sequence.
Let $\mu:=\mu(\mathbf{u}(P))=(\mu_1,\mu_2,\ldots,\mu_{bn})$ 
be the $b$-Stirling permutation obtained from $\mathbf{u}(P)$.

We define $b$ permutations of size $n$ denoted by 
$\nu^{i}:=(\nu^i_1,\ldots,\nu^i_{n})$, $1\le i\le b$, from $\mu$
by
\begin{align}
\label{eqn:nu}
\nu^{i}_{j}:=\mu_{(j-1)b+i},
\end{align}
for $1\le j\le n$.

\begin{lemma}
Eqn. (\ref{eqn:nu}) is well-defined, {\it i.e.}, $\nu^{i}$ is 
a permutation.
\end{lemma}
\begin{proof}
In a $b$-Stirling permutation $\mu$, each integer in $[1,n]$ appears 
exactly $b$ times. 
By definition, we may several integers larger than $i$ between two integer $i$'s. 
By construction, the number of such integers is zero modulo $b$. 
We take integers separated by $b$ steps in $\mu$ to obtain $\nu^{i}$ 
in Eqn. (\ref{eqn:nu}), which insures that $\nu^{i}$ is a permutation.
\end{proof}

We construct $b$ integer sequences 
$\xi^{i}$, $1\le i\le b$, from permutations $\nu^{i}$ by 
the map $\mu^{-1}$, {\it i.e.}, 
we define 
\begin{align}
\label{eqn:xi}
\xi^{i}=\mu^{-1}(\nu^{i}),\quad \forall i\in[1,b].
\end{align}

\begin{lemma}
The integer sequences $\xi^{i}$, $1\le i\le b$, are 
weakly increasing. 
\end{lemma}
\begin{proof}
Let $\mu^{(i)}$ be the $b$-Stirling permutation consisting 
of integers in $[1,i]$.
Since $\mathbf{u}(P)$ is the step sequence of $P$, it 
is a non-decreasing integer sequence. 
This implies that integers $i-1$ and $i$ are next to each other 
in $\mu^{(i)}$ if $i$ is left to $i-1$, or $i$ is right to $i-1$.
From Eqn. (\ref{eqn:nu}), it is obvious that the relative positions 
of integers $i$ and $i-1$ in $\mu$ are preserved. 
More precisely, the integer $i$ is right next to $i-1$, or $i$ is right to $i-1$
if we ignore the integers larger than $i$. 
Since $\xi^{k}$, $1\le k\le b$, is obtained by Eqn. (\ref{eqn:xi}), 
the constraint on $i$ and $i-1$ in $\nu_{k}$ insures that 
$\xi^{k}$ is weakly increasing.
\end{proof}

Let $p_{i}$ be Dyck paths in $\mathfrak{D}_{n}^{(a,1)}$ 
obtained from $P$ by the horizontal strip-decomposition 
as in Definition \ref{defn:hsdecomp}.
\begin{prop}
\label{prop:xiss}
The weakly increasing sequences $\xi^{i}$, $1\le i\le b$, 
are the step sequences of $p_{i}$.
\end{prop}
\begin{proof}
From Proposition \ref{prop:vssp}, it is enough to show that 
$\xi^{i}$ coincides with $\mathbf{v}_{i}$ for $1\le i\le b$.
However, this is clear form the constructions of $\mu$ and 
$\nu^{i}$.
\end{proof}

We apply a vertical strip-decomposition to 
$\xi^{i}:=(\xi^{i}_1,\ldots,\xi^{i}_{an})$, $1\le i\le b$.
We define $\rho^{i,j}:=(\rho^{i,j}_{1},\ldots,\rho^{i,j}_{n})$
with $1\le i\le b$ and $1\le j\le a$ from $\xi^{i}$ by 
\begin{align*}
\rho^{i,j}_{k}:=\xi^{i}_{a(k-1)+j}, 
\end{align*}
for $k\in[1,n]$.

Let $\vartheta(P):=(\vartheta_{i,j})$ for $1\le i\le b$ and $1\le j\le a$
be a Dyck path decomposition of $P$. 
\begin{prop}
The $ab$ weakly increasing sequences $\rho^{i,j}$ for $1\le i\le b$ and 
$1\le j\le a$ are the step sequences of $\vartheta_{i,j}$.
\end{prop}
\begin{proof}
As in the proof of Proposition \ref{prop:transss}, the horizontal 
and vertical strip decompositions are dual to each other.
By applying Proposition \ref{prop:xiss} to the transposed vertical 
strip decomposition, we obtain weakly increasing sequences $\rho^{i,j}$ 
which are the step sequences of $\vartheta_{i,j}$.
\end{proof}

\begin{example}
Let $P=NENENE^2NE^2\in\mathfrak{D}_{2}^{(2,3)}$.
Since $\mu=123344432211$, we have three permutations 
\begin{align*}
\nu^{1}=1342, \quad \nu^2=2431, \quad \nu^3=3421.
\end{align*} 
By applying $\mu^{-1}$ to $\nu^{i}$'s, we have 
\begin{align*}
\xi^1=0112, \quad \xi^2=0011, \quad \xi^3=0001.
\end{align*}
Finally, $\rho^{i,j}$ are given by 
\begin{align*}
\begin{matrix}
\rho^{1,1}=01, & \rho^{2,1}=01, & \rho^{3,1}=00, \\
\rho^{1,2}=12, & \rho^{2,2}=01, & \rho^{3,2}=01.
\end{matrix}
\end{align*}
The integer sequences $\rho^{i,j}$ are the step sequence 
of $\vartheta(P)$ in Example \ref{ex:hvdecomp}.
\end{example}

\subsection{\texorpdfstring{$(a,b)$}{(a,b)}-Dyck paths and Dyck path decomposition}
\label{sec:dpd}
Let $\pi$ be an $(a,b)$-Dyck path.
Let $(p_1,\ldots,p_{b})$ be a horizontal strip decomposition of $\pi$.
We consider a vertical strip decomposition of $p_{i}$, and 
denote it by $(r_{i,1},\ldots,r_{i,a})$ for $1\le i\le b$.
We place $ab$ paths $r_{i,j}$ with $1\le i\le b$ and $1\le j\le a$ as 
\begin{align}
\label{eqn:r}
\begin{matrix}
r_{1,1} & r_{2,1} & \ldots  & r_{b,1} \\
r_{1,2} & r_{2,2} & \ldots  & r_{b,2} \\
 \vdots & \vdots  & \ddots  & \vdots \\
r_{1,a} & r_{2,a} & \ldots  & r_{b,a} \\
\end{matrix}
\end{align}
Then, the Dyck path decomposition is given by 
\begin{align*}
\vartheta:\pi\mapsto \{r_{i,j}: 1\le i\le b, 1\le j\le a\}.
\end{align*}

\begin{remark}
The paths $r_{i,j}$ for $i\in[1,b]$ and $j\in[1,a]$ are all lattice paths 
from $(0,0)$ to $(n,n)$.
By embedding these paths in the set paths of larger size  
by adding $N^{m}$ from left and $E^{m}$ from right for some $m>0$, 
one can regard these paths as Dyck paths of size $m+n$.
This is the reason why we call $\vartheta$ Dyck path decomposition. 
\end{remark}

\begin{prop}
Let $r_{i,j}$ be a path obtained by Dyck path decomposition of $\pi$.
Then, we have 
$r_{i,j}\le r_{i+1,j}$ and $r_{i,j+1}\le r_{i,j}$.
\end{prop}
\begin{proof}
The paths $p_{i}$, $1\le i\le b$, obtained by the horizontal strip
decomposition of $\pi$, satisfy $p_i\le p_{i+1}$ since 
the height sequence of $\pi$ is weakly increasing. 
The paths $r_{i,j}$, $1\le j\le a$, are obtained by the 
vertical strip decomposition of $p_{i}$. 
The step sequence of $p_{i}$ is also weakly increasing, 
which implies that $r_{i,j}\le r_{i+1,j}$ and $r_{i,j+1}\le r_{i,j}$.
\end{proof}

We introduce two relations $\preceq_{h}$ and $\preceq_{v}$ between 
two paths as follows.
Let $r$ and $s$ be two paths from $(0,0)$ to $(n,n)$ and $s$
is above $r$.

Let $\mathbf{u}(r):=(u_1,\ldots,u_{n})$ be the step sequence of $r$.
Then, we define a path $r'$ whose step sequence  
$\mathbf{u}(r'):=(u'_{1},\ldots,u'_{n})$ is given by 
\begin{align}
\label{eqn:udash}
u'_{i}:=
\begin{cases}
0 ,& \text{if } u_{i}=0, \\
u_{i}-1 & \text{if } u_{i}\ge1 \text{\ and\ } u_{i-1}<u_{i} .
\end{cases}
\end{align}
for some $i\in[2,n]$.	
Then, we define 
\begin{align*}
r\preceq_{h}s\Rightarrow r\le s\le r'.
\end{align*}

Similarly, let $\mathbf{h}(r)$ be the height sequence of $r$.
Then, we define a path $r''$ whose height sequence 
$\mathbf{h}(r''):=(h^{''}_{1},\ldots,h^{''}_{n})$ is given by
\begin{align}
\label{eqn:hdash}
h^{''}_{i}:=
\begin{cases}
n, & \text{if } h_{i}=n, \\
h_{i}+1, & \text{if } h_{i}\le n \text{\ and \ } h_{i}<h_{i+1}.
\end{cases}
\end{align}
for some $i\in[1,n-1]$.
Then, we define 
\begin{align*}
r\preceq_{v}s\Rightarrow r\le s\le r''.
\end{align*}

We impose the two relations $\preceq_{h}$ and $\preceq_{v}$ between 
adjacent elements in (\ref{eqn:r}).
More precisely, 
we impose $r_{i,j}\preceq_{h}r_{i+1,j}$ for $i\in[1,b-1]$ and $j\in[1,a]$
and $r_{i,j+1}\preceq_{v}r_{i,j}$ for $i\in[1,b]$ and $j\in[1,a-1]$.
Furthermore, we also impose  
$r_{1,j}\preceq_{h}r_{b,j}$ for $j\in[1,a]$ and 
$r_{i,a}\preceq_{v}r_{i,1}$ for $i\in[1,b]$.

\begin{defn}
We denote by $R(\pi)$ the set of $\{r_{i,j}: 1\le i\le b, 1\le j \le a\}$ 
with relations $\preceq_{h}$ and $\preceq_{v}$ as above.
\end{defn}

\begin{remark}
The relations $\preceq_h$ and $\preceq_v$ come from the strip decomposition 
of a Dyck path.
In the strip decompositions, we take every three entries in the step or the 
height sequences.
This means that the number of the entries with the same value in $r_{i,j}$ 
differs at most one. Therefore, we have the relations for $r_{i,j}$.
\end{remark}

Let $\pi'$ be an $(a,b)$-Dyck path satisfying $\pi\le\pi'$.
Suppose the Dyck path decomposition of $\pi'$ gives 
the set $R(\pi')=\{r'_{i,j}: 1\le i\le b, 1\le j \le a\}$ with 
relations.

If $r_{i,j}\le r'_{i,j}$ for all $i\in[1,b]$ and $j\in[1,a]$, 
we denote by $R(\pi)\preceq R(\pi')$.

\begin{prop}
\label{prop:Dp}
Let $\pi$ be an $(a,b)$-Dyck path.
Then, the cardinality of the set  
\begin{align}
\label{eqn:pidash}
\{\pi': R(\pi)\preceq R(\pi')\},
\end{align}
gives the number of $(a,b)$-Dyck paths above $\pi$.
\end{prop}
\begin{proof}
Two conditions Eqn. (\ref{eqn:udash}) and Eqn. (\ref{eqn:hdash}) 
correspond to deleting a single box from the Young diagram 
characterized by $r$.
The new diagram is also a Young diagram characterized by 
$r'$ or $r''$.
Since we have a natural bijection between Young diagrams 
and rational Dyck paths, 
an element of the set defined in Eqn. (\ref{eqn:pidash}) 
corresponds to a rational Dyck paths $\pi'$ above $\pi$.

Conversely, if we have a rational Dyck path $\pi'$ above $\pi$,
there is a unique Young diagram corresponding to $\pi'$.
This Young diagram satisfy Eqn. (\ref{eqn:udash}) or 
Eqn. (\ref{eqn:hdash}).
\end{proof}

\subsection{Dyck tilings and Dyck path decomposition}
\label{sec:DTDpd}
In Section \ref{sec:dpd}, we consider $(a,b)$-Dyck paths 
above the path $\pi$.
In this subsection, we consider a non-trivial Dyck tiling
above $\pi$.
Here, ``non-trivial Dyck tiling" means that the tiling
contains at least one rational Dyck tiles of size $n\ge1$.
Recall that given a path $\pi$, we have the set of $ab$ 
paths, denoted by $R(\pi)$.

Let $\pi$ be a path and $Y(\pi)$ be the Young diagram characterized by the path $\pi$.
A path $r_{i,j}\in R(\pi)$ can be identified with 
the Young diagram $Y(r_{i,j})$ which is above $r_{i,j}$, right to the line $x=0$,
and below the line $y=n$.
Then, by definition of the Dyck path decomposition of $\pi$, 
there is an obvious bijection between the boxes in $Y(\pi)$
and the boxes in $Y(r_{i,j})$ for $(i,j)\in[1,b]\times[1,a]$.

Let $r$ be a lattice path from $(0,0)$ to $(n,n)$.
Suppose that $r$ contains a partial path $NE$, namely, 
a path consisting of three points $(x,y)$, $(x,y+1)$ and 
$(x+1,y+1)$.
Then, the lattice point $(x,y+1)$ is said to be 
a {\it peak} of $r$. 

Let $v$ be a peak of $r_{i,j}\in R(\pi)$, and $(x,y)$ 
is its coordinate.
Let $b(v)$ be the box in $Y(r_{i,j})$ whose center 
is $(x-1/2,y+1/2)$ if it exists.
By the bijection between the boxes in $Y(\pi)$ and $Y(r_{i,j})$,
we have a box corresponding to $b(v)$, denoted by $\overline{b(v)}$.
Note that the box $\overline{b(v)}$ is in $Y(\pi)$.
We denote by $(\overline{x}-1/2,\overline{y}+1/2)$ the coordinate of the 
center of the box $\overline{b(v)}$. 
We consider the set of lattice points $(\overline{x}+m,\overline{y}-m)$ 
with $m>0$.
There exists a unique lattice points $(\overline{x}+m,\overline{y}-m)$ 
with some $m$, which is on the $(a,b)$-Dyck path $\pi$.
If such a lattice point is a peak of $\pi$, we say that 
the lattice point $(x,y)$ (or equivalently the peak $v$) in $r_{i,j}$ 
is {\it admissible}.
Peaks which is not admissible are called {\it non-admissible} peak.

Let $\pi'$ be a path satisfying $\pi\le\pi'$.
The two paths $r_{i,j}\in R(\pi)$ and $r'_{i,j}\in R(\pi')$ 
determines a skew shape $S_{i,j}$ since $r_{i,j}\le r'_{i,j}$ from 
Proposition \ref{prop:Dp}.
We consider a $(1,1)$-Dyck tiling in the skew shape $S_{i,j}$.
Let $v$ be a non-admissible peak in $r_{i,j}$ and $(v_x,v_y)$ 
be its coordinate.
Then, we consider the condition on Dyck tilings in $S_{i,j}$ as follows:
\begin{enumerate}
\item[($\heartsuit$)] 
The box $b'$ whose center is $(v_x-1/2,y_y+1/2)$ is contained 
in a single box in $S_{i,j}$. This is equivalent to that the box $b'$
cannot be contained in a Dyck tile of size $n\ge1$.
\end{enumerate}

Let $D_{i,j}$ be a $(1,1)$-Dyck tilings in $S_{i,j}$ satisfying 
the condition ($\heartsuit$).
Then, we have the following theorem.
\begin{theorem}
The number of $(a,b)$-Dyck tilings above $\pi$ and below $\pi'$ with 
at least one non-trivial Dyck tiles is 
equal to the cardinality of the set 
\begin{align*}
\mathcal{D}:=\{ (D_{1,1},\ldots, D_{a,b}): \text{ at least one } D_{i,j} \text{ contain a non-trivail Dyck tile}
\}, 
\end{align*}
where $D_{i,j}$ is a $(1,1)$-Dyck tiling in $S_{i,j}$.
Further, there exists a bijection between non-trivial $(a,b)$-Dyck tilings above $\pi$
and the elements in $\mathcal{D}$.
\end{theorem}
\begin{proof}
Recall that we have a natural bijection between a box in $Y(\pi)$ and 
a box in $r_{i,j}$ for some $i\in[1,b]$ and $j\in[1,a]$.

Suppose that we have no non-trivial Dyck tiles in $Y(\pi)$. 
Since all Dyck tiles above $\pi$ and below $\pi'$ are single boxes,
all Dyck tiles in $r_{i,j}$ are single boxes as well.

Suppose that we have at least one non-trivial Dyck tiles in $Y(\pi)$.
Recall that the shape of a non-trivial Dyck tile $d$ is an $(a,b)$-enlarged 
path.
There exists at least one pair $(i,j)$ such that $r_{i,j}$ contains 
at least three boxes corresponding to boxes in $d$.
These three boxes are a left-top box $b_0$ in $d$, a box $b$ step right to $b_0$,
and a box $a$ step below $b_{0}$ in $d$. 
In $r_{i,j}$, these three boxes are next to each other, and form a 
Dyck tile of size one.
If $r_{i,j}$ contains more than three boxes corresponding to boxes 
in $d$, it has a Dyck tile of size larger than one.
From these observations, we have at least one non-trivial Dyck tile
in $r_{i,j}$ for some $i$ and $j$ if a Dyck tiling in $Y(\pi)$ contains 
a non-trivial $(a,b)$-Dyck tile.
By construction, if this map from a Dyck tiling in $Y(\pi)$ to a Dyck 
tiling in $r_{i,j}$ is injective.

Conversely, if $r_{i,j}$ contains a non-trivial Dyck tile for some 
$i$ and $j$, we have at least one $(a,b)$-Dyck tiles in $Y(\pi)$ 
by the correspondence between boxes in $r_{i,j}$ and in $Y(\pi)$.
Again, this map from a Dyck tiling in $r_{i,j}$ to a Dyck tiling 
in $Y(\pi)$ is injective.

We have a bijection between a Dyck tiling in $Y(\pi)$ and 
$ab$ Dyck tilings $(D_{1,1},\ldots,D_{a,b})$. 
This completes the proof.
\end{proof}

\section{Weight of a cover-inclusive rational Dyck tiling}
\label{sec:weight}
In this section, we study the relation between a $(a,b)$-Dyck tilings and 
strip decompositions of them in view of the weights of Dyck tilings.
First, we consider the horizontal decomposition of $(1,b)$-Dyck tilings,
then deal with the vertical decomposition of $(a,1)$-Dyck tlings.
Finally, we apply these two cases to the general $(a,b)$-Dyck tlings.

In Sections \ref{sec:dpd} and \ref{sec:DTDpd}, we have an description 
of an $(a,b)$-Dyck tiling $\mathcal{D}$ in terms of the Dyck 
path decomposition $(r_{i,j})$ with $1\le i\le b$ and $1\le j\le a$.
There, we have used the bijection between boxes in $\mathcal{D}$
and boxes in $r_{i,j}$.
In this section, we give a different description of $\mathcal{D}$ 
in terms of $r_{i,j}$.
We make use of an Hermite history of $\mathcal{D}$ to establish 
the correspondence between $\mathcal{D}$ and $r_{i,j}$.

\subsection{\texorpdfstring{$(1,b)$}{(1,b)}-Dyck tilings and weight of a Dyck tiling}
In this subsection, we study rational Dyck tilings for $(a,b)=(1,b)$, and 
its relation to $b$-Stirling permutations.
We consider the horizontal strip decomposition of Dyck tilings.
Let $P$ be a rational Dyck path $\mathfrak{D}_{n}^{(1,b)}$ and 
$\mathfrak{p}:=(p_1,\ldots,p_b)$ be lattice paths obtained from the horizontal 
strip decomposition.
We denote by $\mathcal{D}$ a $(1,b)$-Dyck tilling above $P$.
From Section \ref{sec:dpd}, the lattice paths $p_i$, $1\le i\le b$, satisfy 
$p_i\preceq_{h}p_{i+1}$ for $1\le i\le b-1$ and $p_1\preceq_{h}p_{b}$.

Let $\mathcal{D}_i$ be a Dyck tiling above $p_{i}$, and $\nu(\mathcal{D}_{i})$ 
be the post-order word obtained from $\mathcal{D}_{i}$ by the DTS bijection 
(see Section \ref{sec:DTS}).
We denote $t_i$ be a top path of the Dyck tiling $\mathcal{D}_{i}$.
Since the top paths $t_i$, $1\le i\le b$, also form a top path of $\mathcal{D}$,
we have $t_i\preceq_{h}t_{i+1}$. 
We denote by $\le$ the lexicographic order of words. The order for alphabets 
is $1<2<3\ldots<n$.
\begin{lemma}
\label{lemma:hnu}
Suppose $p_i\preceq_{h}p_{i+1}$ and $t_{i}=t_{i+1}$.
We have $\nu(\mathcal{D}_{i+1})\le \nu(\mathcal{D}_{i})$ if both Dyck tilings consist
of single Dyck tiles.
\end{lemma}
\begin{proof}

Let $\mathfrak{h}:=(h_1,\ldots,h_{n})$ 
be an integer sequences such that $h_{j}$ is the label on the $j$-th chord in an Hermite history.
We denote by $\mathfrak{h}$ and $\mathfrak{h'}$ the integer sequences for $\mathcal{D}_{i}$
and $\mathcal{D}_{i}$ respectively.

The condition that $p_{i}\preceq_{h}p_{i+1}$ means that we have some $j$
such that $h_{j}'=h_{j}-1$. 
This is because that the top paths of $\mathcal{D}_{i}$ and $\mathcal{D}_{i+1}$ are 
the same.  
Since we read the word by post-order, we have the same word, or 
we exchange $k$ and $k'$ in the Hermite history such that 
$k'$ is the maximal integer smaller than $k$ and left to $k$.
From Proposition \ref{prop:dualHhDTS}, a chord has a label $n+1-j$ in the 
DTS bijection if the same chord has a label $j$ in the Hermite history.
Thus, if we read the post-order word in the DTS bijection, 
$\nu(\mathcal{D}_{i})\ge\nu(\mathcal{D}_{i+1})$.
\end{proof}

\begin{example}
In Figure \ref{fig:DT1}, we show some examples of Dyck tiling with the top path $N^3ENE^3$ and 
without non-trivial Dyck tiles.
\begin{figure}[ht]
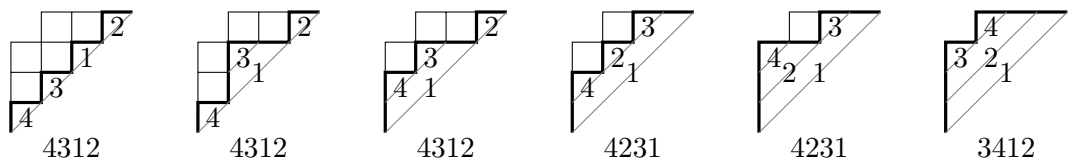

\tikzpic{-0.5}{[scale=0.4]
\draw[very thick](0,0)--(0,1)--(1,1)--(1,2)--(2,2)--(2,3)--(3,3)--(3,4)--(4,4);
\draw(0,1)--(0,3)--(2,3)--(2,4)(0,2)--(1,2)--(1,4)--(3,4);
\draw[gray](0,0)--(4,4);
\draw(0.5,0.5)node{$4$}(1.5,1.5)node{$3$}(2.5,2.5)node{$1$}(3.5,3.5)node{$2$};
\draw(2,-0.5)node{$4312$};
}\quad
\tikzpic{-0.5}{[scale=0.4]
\draw[very thick](0,0)--(0,1)--(1,1)--(1,3)--(3,3)--(3,4)--(4,4);
\draw(0,1)--(0,3)--(1,3)--(1,4)--(3,4)(0,2)--(1,2)(2,3)--(2,4);
\draw[gray](0,0)--(4,4)(1,2)--(2,3);
\draw(0.5,0.5)node{$4$}(1.5,2.5)node{$3$}(2,2)node{$1$}(3.5,3.5)node{$2$};
\draw(2,-0.5)node{$4312$};
}\quad
\tikzpic{-0.5}{[scale=0.4]
\draw[very thick](0,0)--(0,2)--(1,2)--(1,3)--(3,3)--(3,4)--(4,4);
\draw(0,2)--(0,3)--(1,3)--(1,4)--(3,4)(2,3)--(2,4);
\draw[gray](0,0)--(4,4)(0,1)--(2,3);
\draw(1.5,1.5)node{$1$}(1.5,2.5)node{$3$}(0.5,1.5)node{$4$}(3.5,3.5)node{$2$};
\draw(2,-0.5)node{$4312$};
}\quad
\tikzpic{-0.5}{[scale=0.4]
\draw[very thick](0,0)--(0,2)--(1,2)--(1,3)--(2,3)--(2,4)--(4,4);
\draw(0,2)--(0,3)--(1,3)--(1,4)--(2,4);
\draw[gray](0,0)--(4,4)(0,1)--(3,4);
\draw(2,2)node{$1$}(0.5,1.5)node{$4$}(1.5,2.5)node{$2$}(2.5,3.5)node{$3$};
\draw(2,-0.5)node{$4231$};
}\quad
\tikzpic{-0.5}{[scale=0.4]
\draw[very thick](0,0)--(0,3)--(2,3)--(2,4)--(4,4);
\draw(1,3)--(1,4)--(2,4);
\draw[gray](0,0)--(4,4)(0,1)--(3,4)(0,2)--(1,3);
\draw(0.5,2.5)node{$4$}(1,2)node{$2$}(2,2)node{$1$}(2.5,3.5)node{$3$};
\draw(2,-0.5)node{$4231$};
}\quad
\tikzpic{-0.5}{[scale=0.4]
\draw[very thick](0,0)--(0,3)--(1,3)--(1,4)--(4,4);
\draw[gray](0,0)--(4,4)(0,1)--(3,4)(0,2)--(2,4);
\draw(0.5,2.5)node{$3$}(1.5,2.5)node{$2$}(2,2)node{$1$}(1.5,3.5)node{$4$};
\draw(2,-0.5)node{$3412$};
}
\caption{Examples of labels for Dyck tlings with the fixed top path.}
\label{fig:DT1}
\end{figure}
\end{example}

\begin{remark}
We remark that if a Dyck tiling contains Dyck tiles of size larger than 
zero, we have similar properties to Lemma \ref{lemma:hnu} (see Figure \ref{fig:ntDt}).
However, the comparison between a Dyck tiling without non-trivial Dyck tiles
and a Dyck tiling with non-trivial Dyck tiles is not obvious.
\end{remark}

\begin{example}
We consider the $(1,3)$-Dyck path $\lambda$ with the step sequence $(0,2,2,8)$.
We have eight non-trivial Dyck tilings.
For example, three of them are depicted in Figure \ref{fig:hsdecomo}.
\begin{figure}[ht]
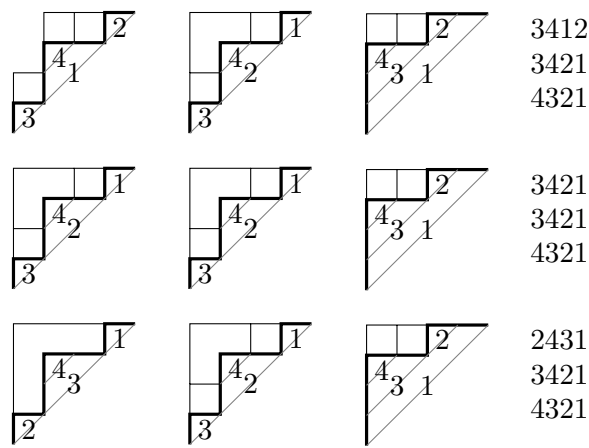

\begin{align*}
\tikzpic{-0.5}{[scale=0.4]
\draw[very thick](0,0)--(0,1)--(1,1)--(1,3)--(3,3)--(3,4)--(4,4);
\draw(0,1)--(0,2)--(1,2)(1,3)--(1,4)--(3,4)(2,3)--(2,4);
\draw[gray](0,0)--(4,4)(1,2)--(2,3);
\draw(0.5,0.5)node{$3$}(1.5,2.5)node{$4$}(2,2)node{$1$}(3.5,3.5)node{$2$};
}\quad
\tikzpic{-0.5}{[scale=0.4]
\draw[very thick](0,0)--(0,1)--(1,1)--(1,3)--(3,3)--(3,4)--(4,4);
\draw(0,1)--(0,2)--(1,2)(2,3)--(2,4)--(3,4)(0,2)--(0,4)--(2,4);
\draw[gray](0,0)--(4,4)(1,2)--(2,3);
\draw(0.5,0.5)node{$3$}(1.5,2.5)node{$4$}(2,2)node{$2$}(3.5,3.5)node{$1$};
}\quad
\tikzpic{-0.5}{[scale=0.4]
\draw[very thick](0,0)--(0,3)--(2,3)--(2,4)--(4,4);
\draw(0,3)--(0,4)--(2,4)(1,3)--(1,4);
\draw[gray](0,0)--(4,4)(0,1)--(3,4)(0,2)--(1,3);
\draw(0.5,2.5)node{$4$}(1,2)node{$3$}(2.5,3.5)node{$2$}(2,2)node{$1$};
}\quad
\begin{matrix}
3412 \\
3421 \\
4321
\end{matrix}
\end{align*}
\begin{align*}
\tikzpic{-0.5}{[scale=0.4]
\draw[very thick](0,0)--(0,1)--(1,1)--(1,3)--(3,3)--(3,4)--(4,4);
\draw(0,1)--(0,2)--(1,2)(2,3)--(2,4)--(3,4)(0,2)--(0,4)--(2,4);
\draw[gray](0,0)--(4,4)(1,2)--(2,3);
\draw(0.5,0.5)node{$3$}(1.5,2.5)node{$4$}(2,2)node{$2$}(3.5,3.5)node{$1$};
}\quad
\tikzpic{-0.5}{[scale=0.4]
\draw[very thick](0,0)--(0,1)--(1,1)--(1,3)--(3,3)--(3,4)--(4,4);
\draw(0,1)--(0,2)--(1,2)(2,3)--(2,4)--(3,4)(0,2)--(0,4)--(2,4);
\draw[gray](0,0)--(4,4)(1,2)--(2,3);
\draw(0.5,0.5)node{$3$}(1.5,2.5)node{$4$}(2,2)node{$2$}(3.5,3.5)node{$1$};
}\quad
\tikzpic{-0.5}{[scale=0.4]
\draw[very thick](0,0)--(0,3)--(2,3)--(2,4)--(4,4);
\draw(0,3)--(0,4)--(2,4)(1,3)--(1,4);
\draw[gray](0,0)--(4,4)(0,1)--(3,4)(0,2)--(1,3);
\draw(0.5,2.5)node{$4$}(1,2)node{$3$}(2.5,3.5)node{$2$}(2,2)node{$1$};
}\quad
\begin{matrix}
3421\\
3421 \\
4321
\end{matrix}
\end{align*}
\begin{align*}
\tikzpic{-0.5}{[scale=0.4]
\draw[very thick](0,0)--(0,1)--(1,1)--(1,3)--(3,3)--(3,4)--(4,4);
\draw(0,1)--(0,4)--(3,4);
\draw[gray](0,0)--(4,4)(1,2)--(2,3);
\draw(0.5,0.5)node{$2$}(1.5,2.5)node{$4$}(2,2)node{$3$}(3.5,3.5)node{$1$};
}\quad
\tikzpic{-0.5}{[scale=0.4]
\draw[very thick](0,0)--(0,1)--(1,1)--(1,3)--(3,3)--(3,4)--(4,4);
\draw(0,1)--(0,2)--(1,2)(2,3)--(2,4)--(3,4)(0,2)--(0,4)--(2,4);
\draw[gray](0,0)--(4,4)(1,2)--(2,3);
\draw(0.5,0.5)node{$3$}(1.5,2.5)node{$4$}(2,2)node{$2$}(3.5,3.5)node{$1$};
}\quad
\tikzpic{-0.5}{[scale=0.4]
\draw[very thick](0,0)--(0,3)--(2,3)--(2,4)--(4,4);
\draw(0,3)--(0,4)--(2,4)(1,3)--(1,4);
\draw[gray](0,0)--(4,4)(0,1)--(3,4)(0,2)--(1,3);
\draw(0.5,2.5)node{$4$}(1,2)node{$3$}(2.5,3.5)node{$2$}(2,2)node{$1$};
}\quad
\begin{matrix}
2431\\
3421 \\
4321
\end{matrix}
\end{align*}
\caption{Dyck paths and permutations}
\label{fig:hsdecomo}
\end{figure}
\end{example}

Let $p$ be a $(1,1)$-Dyck path. We consider $(1,1)$-Dyck tilings with 
non-trivial Dyck tiles. 
Suppose that a Dyck tiling $D_0$ contains a Dyck tile of size $m$, and 
a Dyck tiling $D_1$ contains a Dyck tile of size $m'$ with $m'<m$.
Then, the post-order words $\nu(D_0)$ and $\nu(D_1)$ satisfy 
$\nu(D_0)<\nu(D_1)$ in the lexicographic order.
Figure \ref{fig:ntDt} shows an example.
\begin{figure}[ht]
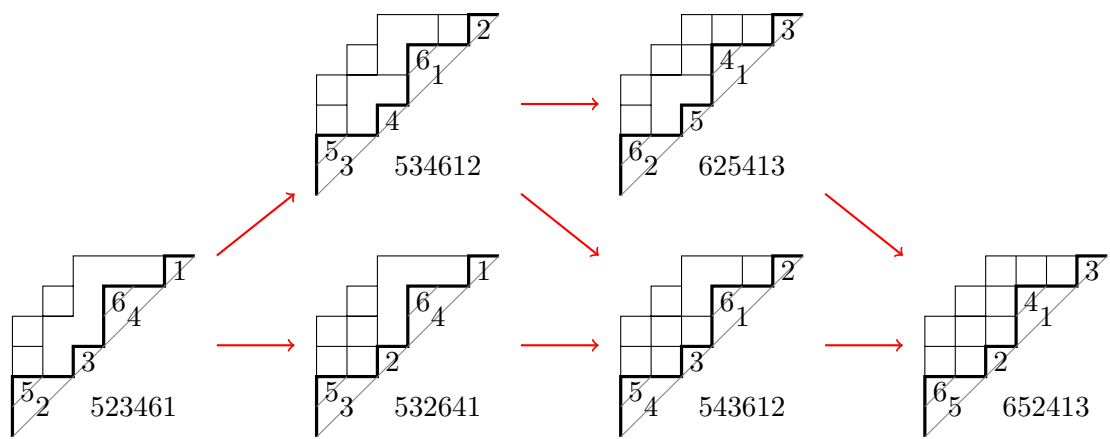

\tikzpic{-0.5}{[scale=4]
\node (origin) at (0,0) {
\tikzpic{-0.5}{[scale=0.4]
\draw[very thick](0,0)--(0,2)--(2,2)--(2,3)--(3,3)--(3,5)--(5,5)--(5,6)--(6,6);
\draw(0,2)--(0,4)--(2,4)--(2,6)--(5,6)(1,2)--(1,4)--(2,4)(0,3)--(1,3)(1,4)--(1,5)--(2,5);
\draw[gray](0,0)--(6,6)(0,1)--(1,2)(3,4)--(4,5);
\draw(1,1)node{$2$}(0.5,1.5)node{$5$}(2.5,2.5)node{$3$}(4,4)node{$4$}(3.5,4.5)node{$6$}
(5.5,5.5)node{$1$};
\draw(4,1)node{$523461$};
}
};
\node (dg2) at (1,0.8) {
\tikzpic{-0.5}{[scale=0.4]
\draw[very thick](0,0)--(0,2)--(2,2)--(2,3)--(3,3)--(3,5)--(5,5)--(5,6)--(6,6);
\draw(0,2)--(0,4)--(2,4)--(2,6)--(5,6)(1,2)--(1,4)--(2,4)(0,3)--(1,3)(1,4)--(1,5)--(2,5);
\draw(2,4)--(3,4)(4,5)--(4,6);
\draw[gray](0,0)--(6,6)(0,1)--(1,2)(3,4)--(4,5);
\draw(1,1)node{$3$}(0.5,1.5)node{$5$}(2.5,2.5)node{$4$}(4,4)node{$1$}(3.5,4.5)node{$6$}
(5.5,5.5)node{$2$};
\draw(4,1)node{$534612$};
}
};
\node (dg3) at (2,0.8) {
\tikzpic{-0.5}{[scale=0.4]
\draw[very thick](0,0)--(0,2)--(2,2)--(2,3)--(3,3)--(3,5)--(5,5)--(5,6)--(6,6);
\draw(0,2)--(0,4)--(2,4)--(2,6)--(5,6)(1,2)--(1,4)--(2,4)(0,3)--(1,3)(1,4)--(1,5)--(2,5);
\draw(2,4)--(3,4)(2,5)--(3,5)(3,5)--(3,6)(4,5)--(4,6);
\draw[gray](0,0)--(6,6)(0,1)--(1,2)(3,4)--(4,5);
\draw(1,1)node{$2$}(0.5,1.5)node{$6$}(2.5,2.5)node{$5$}(4,4)node{$1$}(3.5,4.5)node{$4$}
(5.5,5.5)node{$3$};
\draw(4,1)node{$625413$};
}
};
\node (dg4) at (2,0) {
\tikzpic{-0.5}{[scale=0.4]
\draw[very thick](0,0)--(0,2)--(2,2)--(2,3)--(3,3)--(3,5)--(5,5)--(5,6)--(6,6);
\draw(0,2)--(0,4)--(2,4)--(2,6)--(5,6)(1,2)--(1,4)--(2,4)(0,3)--(1,3)(1,4)--(1,5)--(2,5);
\draw(1,3)--(2,3)--(2,4)--(3,4)(4,5)--(4,6);
\draw[gray](0,0)--(6,6)(0,1)--(1,2)(3,4)--(4,5);
\draw(1,1)node{$4$}(0.5,1.5)node{$5$}(2.5,2.5)node{$3$}(4,4)node{$1$}(3.5,4.5)node{$6$}
(5.5,5.5)node{$2$};
\draw(4,1)node{$543612$};
}
};
\node (dg5) at (1,-0) {
\tikzpic{-0.5}{[scale=0.4]
\draw[very thick](0,0)--(0,2)--(2,2)--(2,3)--(3,3)--(3,5)--(5,5)--(5,6)--(6,6);
\draw(0,2)--(0,4)--(2,4)--(2,6)--(5,6)(1,2)--(1,4)--(2,4)(0,3)--(1,3)(1,4)--(1,5)--(2,5);
\draw(1,3)--(2,3)--(2,4);
\draw[gray](0,0)--(6,6)(0,1)--(1,2)(3,4)--(4,5);
\draw(1,1)node{$3$}(0.5,1.5)node{$5$}(2.5,2.5)node{$2$}(4,4)node{$4$}(3.5,4.5)node{$6$}
(5.5,5.5)node{$1$};
\draw(4,1)node{$532641$};
}
};
\node (dg6) at (3,0) {
\tikzpic{-0.5}{[scale=0.4]
\draw[very thick](0,0)--(0,2)--(2,2)--(2,3)--(3,3)--(3,5)--(5,5)--(5,6)--(6,6);
\draw(0,2)--(0,4)--(2,4)--(2,6)--(5,6)(1,2)--(1,4)--(2,4)(0,3)--(1,3)(1,4)--(1,5)--(2,5);
\draw(1,3)--(2,3)--(2,4)--(3,4)(2,5)--(3,5)--(3,6)(4,5)--(4,6);
\draw[gray](0,0)--(6,6)(0,1)--(1,2)(3,4)--(4,5);
\draw(1,1)node{$5$}(0.5,1.5)node{$6$}(2.5,2.5)node{$2$}(4,4)node{$1$}(3.5,4.5)node{$4$}
(5.5,5.5)node{$3$};
\draw(4,1)node{$652413$};
}
};
\draw[red,thick,->](origin)--(dg2);
\draw[red,thick,->](origin)--(dg5);
\draw[red,thick,->](dg2)--(dg3);
\draw[red,thick,->](dg3)--(dg6);
\draw[red,thick,->](dg2)--(dg4);
\draw[red,thick,->](dg5)--(dg4);
\draw[red,thick,->](dg4)--(dg6);
}
\caption{Lexicographic order of the post-order words with non-trivial Dyck tiles
with the lowest path $N^2E^2NEN^2E^2NE$.}
\label{fig:ntDt}
\end{figure}
The red lines indicate the order of two elements.

We construct a $(1,b)$-Dyck tiling $\mathcal{D}$ from $b$ Dyck tilings 
$\mathcal{D}_{i}$, $1\le i\le b$.
Recall that an Hermite history consists of trajectories 
starting from the $N$ step in the lowest Dyck path (see Section \ref{sec:Hh}).
We define a non-negative integer sequences $g^{i}:=(g^i_1,\ldots,g^i_n)$, 
where $g^i_j$ is the $l(S_N)-1$ for the $j$-th $N$ step $S_{N}$ in $\mathcal{D}_{i}$.
We define 
\begin{align*}
G_j(\mathcal{D}_1,\ldots,\mathcal{D}_b)
:=\sum_{1\le i\le b}g^{i}_{j},
\end{align*}
where $1\le j \le n$.

We define a $(1,b)$-Dyck tiling $\mathcal{D}$ such that 
the Hermite history of $\mathcal{D}$ is $(G_1,\ldots,G_{n})$.

We abbreviate $\nu(\mathcal{D}_{i})$ as $\nu^{i}$.
We denote by $N(\nu^1,\ldots,\nu^{n})$ the sum of entries in $\mu(\nu^{i})$ for 
$1\le i\le b$. 
Let $Y(P)$ be the sum of entries of step sequence of the lowest 
path $P$.

\begin{prop}
\label{prop:wtD}
We have 
\begin{align}
\label{eqn:wtD}
\mathrm{wt}(\mathcal{D})=Y(P)-N(\nu^1,\ldots,\nu^{n}).
\end{align}
\end{prop}
\begin{proof}
We prove Eqn. (\ref{eqn:wtD}) by induction.
We first show that Eqn. (\ref{eqn:wtD}) holds when 
the Dyck tiling $\mathcal{D}$ has no Dyck tiles.
If we read the words consisting of the labels on chords of the lowest path $P$ 
in $\mathcal{D}$ in the pre-order, we have the identity permutation.  
We use the post-order words to obtain $\nu^{i}$. 
We denote by $\mathbf{u}(p_i):=(u_1,\ldots,u_n)$ the step sequence of $p_i$.
Suppose $u_{i-1}=u_{i}$. 
Then, the integer $i$ appears just before $i-1$ in the post-order words.
Thus in $\nu^{i}$, the insertion history $\mu(\nu^{i}):=(\mu_1,\ldots,\mu_n)$ 
satisfies $\mu_{i-1}=\mu_{i}$.
Suppose $u_{i-1}<u_{i}$. By a similar argument, we have $\mu_{i-1}<\mu_{i}$.
From these observations, $h_{i}$ is the sum of the numbers of boxes 
left to the chord with label $i$.
Therefore, the insertion history is nothing but the step sequence of 
$P$. We obtain $\mathrm{wt}(\mathcal{D})=0$ since $N(\lambda)=Y(P)$.

We assume that Eqn. (\ref{eqn:wtD}) is true for all Dyck tiling $\mathcal{D}'$ 
with the top path $Q'$. We show Eqn. (\ref{eqn:wtD}) is also true for 
$Q$ such that the number of boxes in $Y(Q)$ is one plus the number 
of boxes in $Y(Q')$.
We denote by $\mathcal{D}$ the Dyck tiling in $Y(Q)$. 
In $\mathcal{D}_{i}$ for some $i$, the number of boxes is increased by one,
or the size of a non-trivial Dyck tile is increased by one.
Then, it is obvious that $N(\lambda(\mathcal{D}'))$ is increased by one from the construction 
of the post-order words. We have $N(\lambda(\mathcal{D}))=N(\lambda(\mathcal{D}'))+1$
We obtain $\mathrm{wt}(\mathcal{D})=Y(P)-N(\lambda(\mathcal{D}))=\mathrm{wt}(\mathcal{D})+1$.
This completes the proof.
\end{proof}

\begin{example}
We consider the third example in Figure \ref{fig:hsdecomo}.
The corresponding $(1,3)$-Dyck tiling is depicted as 
\begin{center}
\tikzpic{-0.5}{[scale=0.5]
\draw[very thick](0,0)--(0,1)--(2,1)--(2,3)--(8,3)--(8,4)--(12,4);
\draw(1,1)--(1,4)--(8,4);
\draw(0,1)--(0,3)--(1,3)(0,2)--(1,2);
}
\end{center}
Note that $g^1=(0,0,0,3)$, $g^2=(0,1,0,3)$ and $g^3=(0,0,0,2)$ give
$G=(G_1,G_2,G_3,G_4)=(0,1,0,8)$.
The step sequence of the lowest path is $(0,2,2,8)$, which 
implies $Y(p)=12$.
Since $\mu(2431)=(0,0,1,1)$, $\mu(3421)=(0,0,0,1)$ and 
$\mu(4321)=(0,0,0,0)$, we obtain $N(\nu^1,\nu^2,\nu^3)=3$.
Thus, the weight of the $(1,3)$-Dyck tiling is $12-3=9$.
\end{example}

\subsection{\texorpdfstring{$(a,1)$}{(a,1)}-Dyck tilings and weight of a Dyck tiling}
In this subsection, we study the relation between $(a,1)$-Dyck tilings and 
$a$-Stirling permutations. 
The results are parallel to those of $(1,b)$-Dyck tilings by changing $a$ and $b$.
However, the statistics $\mathrm{art}(\mathcal{D})$ for a Dyck tiling is not 
preserved by the exchange of $a$ and $b$. 

Let $P$ be a rational Dyck path in $\mathfrak{D}_{n}^{(a,1)}$, and 
$\mathfrak{q}:=(q_1,\ldots,q_a)$ be lattice paths obtained by the 
vertical strip decomposition.
From Section \ref{sec:dpd}, lattice paths $q_{i}$, $1\le i\le a$, 
satisfy $q_{i+1}\preceq_{v}q_{i}$ for $1\le i\le b-1$ and 
$q_b\preceq_{v}q_{1}$.

Let $\mathcal{D}_{i}$ be a Dyck tiling above $q_i$, and $\upsilon_{i}:=\upsilon(\mathcal{D}_{i})$
be the post-order word reading from right to left obtained from $\mathcal{D}_{i}$ by the left DTS 
bijection.
The top paths $t_i$ of $\mathcal{D}_{i}$ satisfies $t_{i+1}\preceq_{v}t_{i}$.

By a similar argument to the proof of Lemma \ref{lemma:hnu}, we obtain the following lemma.
\begin{lemma}
Suppose $q_{i+1}\preceq_{v}q_{i}$, and $\mathcal{D}_{i}$ and $\mathcal{D}_{i+1}$ 
have the same top path. 
We have $\upsilon(\mathcal{D}_{i})\le\upsilon(\mathcal{D}_{i+1})$ if neither
$\mathcal{D}_{i}$ and $\mathcal{D}_{i+1}$ have non-trivial Dyck tiles.
\end{lemma}

\begin{remark}
Note that the order is reversed compared to the horizontal strip decomposition.
This is because we make use of left DTS.
We also read the labels of chords from right to left in the post-order.
\end{remark}

Let $N(\upsilon_{1},\ldots,\upsilon_{n})$ be the sum of entries of 
the insertion history $\mu(\upsilon_i)$, $1\le i\le a$.
\begin{prop}
\label{prop:wtDdual}
We have 
\begin{align}
\label{eqn:wtDdual}
\mathrm{wt}(\mathcal{D})=Y(P)-N(\upsilon_{1},\ldots,\upsilon_{n})-M(a-1),
\end{align}
where $M$ is the sum of the sizes of a non-trivial Dyck tiles in $\mathcal{D}$.
\end{prop}
\begin{proof}
The proof is essentially the same as the proof of Proposition \ref{prop:wtD}.
The difference is that the statistic $\mathrm{art}(\mathcal{D})$ for a non-trivial
Dyck tile is not equal to the size of a non-trivial Dyck tile in $\mathcal{D}$.
In Eqn. (\ref{eqn:wtDdual}), this difference is computed by $M(a-1)$.
This completes the proof.
\end{proof}

To obtain the $(a,1)$-Dyck tiling from $a$ $(1,1)$-Dyck tilings, 
we make use of the Hermite history as in case of a $(1,b)$-Dyck tiling.
Here, we have to use the Hermite history attached to the $E$ steps
in the lowest path.

\begin{example}
We consider a $(3,1)$-Dyck tiling as in Figure \ref{fig:a1}.
\begin{figure}[ht]
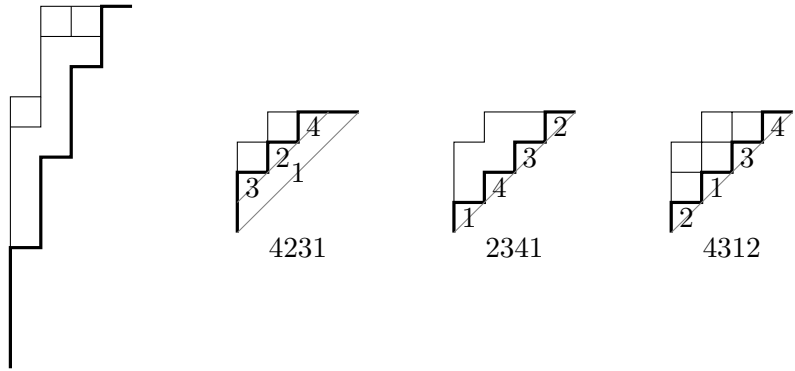

\tikzpic{-0.5}{[scale=0.4]
\draw[very thick](0,0)--(0,4)--(1,4)--(1,7)--(2,7)--(2,10)--(3,10)--(3,12)--(4,12);
\draw(0,4)--(0,9)--(1,9)--(1,12)--(3,12);
\draw(0,8)--(1,8)--(1,9)(1,11)--(3,11)(2,11)--(2,12);
}
\qquad 
\tikzpic{-0.5}{[scale=0.4]
\draw[very thick](0,0)--(0,2)--(1,2)--(1,3)--(2,3)--(2,4)--(4,4);
\draw(0,2)--(0,3)--(1,3)--(1,4)--(2,4);
\draw[gray](0,0)--(4,4)(0,1)--(3,4);
\draw(2,2)node{$1$}(0.5,1.5)node{$3$}(1.5,2.5)node{$2$}(2.5,3.5)node{$4$};
\draw(2,-0.5)node{$4231$};
}\qquad
\tikzpic{-0.5}{[scale=0.4]
\draw[very thick](0,0)--(0,1)--(1,1)--(1,2)--(2,2)--(2,3)--(3,3)--(3,4)--(4,4);
\draw(0,1)--(0,3)--(1,3)--(1,4)--(3,4);
\draw[gray](0,0)--(4,4);
\draw(0.5,0.5)node{$1$}(1.5,1.5)node{$4$}(2.5,2.5)node{$3$}(3.5,3.5)node{$2$};
\draw(2,-0.5)node{$2341$};
}\qquad
\tikzpic{-0.5}{[scale=0.4]
\draw[very thick](0,0)--(0,1)--(1,1)--(1,2)--(2,2)--(2,3)--(3,3)--(3,4)--(4,4);
\draw(0,1)--(0,3)--(1,3)--(1,4)--(3,4)(0,2)--(1,2)--(1,3)--(2,3)--(2,4);
\draw[gray](0,0)--(4,4);
\draw(0.5,0.5)node{$2$}(1.5,1.5)node{$1$}(2.5,2.5)node{$3$}(3.5,3.5)node{$4$};
\draw(2,-0.5)node{$4312$};
}
\caption{Decomposition of a $(3,1)$-Dyck tiling}
\label{fig:a1}
\end{figure}
Since $\upsilon_1=4231$, $\upsilon_{2}=2341$ and $\upsilon_{3}=4312$, 
we have $\mu(\upsilon_1)=(0,0,1,0)$, $\mu(\upsilon_2)=(0,0,1,2)$ and 
$\mu(\upsilon_3)=(0,1,0,0)$. We have $N(\upsilon_1,\upsilon_2,\upsilon_3)=4$
We have $Y(P)=15$, $M=2$ and $a=3$ for the $(3,1)$-Dyck tiling $\mathcal{D}$.
We obtain $\mathrm{wt}(\mathcal{D})=15-5-4=6$.

\end{example}

\subsection{\texorpdfstring{$(a,b)$}{(a,b)}-Dyck tilings}
Let $\mathcal{D}$ be a $(a,b)$-Dyck tiling, 
and $(r_{i,j})$ for $1\le i\le b$ and $1\le j\le a$ be its Dyck path 
decomposition.
Note that each row gives a horizontal strip decomposition of some 
$(1,b)$-Dyck tiling, and each column gives a vertical strip decomposition 
of some $(a,1)$-Dyck tiling.
Therefore, one can apply Propositions \ref{prop:wtD} or \ref{prop:wtDdual}
to rows or columns of $(r_{i,j})$.

\bibliographystyle{amsplainhyper} 
\bibliography{biblio}

\end{document}